\documentclass[a4paper]{amsart}

\usepackage[utf8]{inputenc}  
\usepackage{amsmath}

\usepackage[dvipsnames]{xcolor}
\usepackage[hypertexnames=false,linkcolor=black, colorlinks=true, citecolor=black, filecolor=black,urlcolor=black]{hyperref}

\usepackage{algorithmicx}
\usepackage{algorithm}
\usepackage{algpseudocode}

\algtext*{EndWhile}
\algtext*{EndIf}
\algtext*{EndFor}

\usepackage{eqparbox}

\usepackage{float}
\newfloat{algorithm}{t}{lop}

\usepackage{tikz}
\usetikzlibrary{shapes,arrows}
\tikzstyle{decision} = [diamond, draw,
    text width=5em, text centered, node distance=3cm, inner sep=5pt]
\tikzstyle{block} = [rectangle, draw,
    text width=8em, text centered, rounded corners, minimum height=4em]
\tikzstyle{line} = [draw, -latex']
\tikzstyle{cloud} = [draw, ellipse, node distance=3cm,
    minimum height=2em]

\usepackage{amsfonts,amssymb,amsthm,amscd,epsfig,epstopdf,url,array}

\algnewcommand\algorithmicinput{\textbf{INPUT:}}
\algnewcommand\INPUT{\item[\algorithmicinput]}

\algnewcommand\algorithmicoutput{\textbf{OUTPUT:}}
\algnewcommand\OUTPUT{\item[\algorithmicoutput]}

\numberwithin{equation}{section}

\theoremstyle{plain}
\newtheorem{Thm}{Theorem}[section]
\newtheorem{Lem}[Thm]{Lemma}
\newtheorem{Prop}[Thm]{Proposition}
\newtheorem{Cor}[Thm]{Corollary}

\theoremstyle{definition}
\newtheorem{defi}[Thm]{Definition}
\newtheorem{Rk}[Thm]{Remark}
\newtheorem{example}[Thm]{Example}
\newtheorem{Constr}[Thm]{Construction}
\newtheorem{Obs}[Thm]{Observation}

\theoremstyle{remark}
\newtheorem*{acknowledgements}{Acknowledgements}

\DeclareMathOperator{\inv}{inv}
\DeclareMathOperator{\ord}{ord}
\DeclareMathOperator{\maxord}{max-ord}
\DeclareMathOperator{\Maxord}{Max-ord}
\DeclareMathOperator{\Bl}{Bl}
\DeclareMathOperator{\Sing}{Sing}
\DeclareMathOperator{\Spec}{Spec}

\newcommand{\fa}{\mathfrak{a}}
\newcommand{\fb}{\mathfrak{b}}

\newcommand{\IA}{\mathbb{A}}
\newcommand{\IQ}{\mathbb{Q}}
\newcommand{\IZ}{\mathbb{Z}}

\newcommand{\ux}{\underline{x}}

\newcommand{\cc}{\rho}
\newcommand{\lex}{{\mathrm{lex}}}

\title[Algorithmic local monomialization of binomials]{Algorithmic local monomialization of a binomial: a comparison of different approaches}

\author[Gaube]{Sabrina Alexandra Gaube}
\address{
	Institut f\"ur Mathematik, 
	Carl von Ossietzky Universit\"at Oldenburg, 
	26111 Oldenburg, Germany}
\email{sabrina.gaube@uni-oldenburg.de}
\email{bernd.schober@uni-oldenburg.de}

\author[Schober]{Bernd Schober}
\thanks{The second named author was partially supported by the DFG-project ``Order zeta functions and resolutions of singularities" (DFG project number: 373111162).}

\date{\today}

\makeatletter
\@namedef{subjclassname@2020}{\textup{2020} Mathematics Subject Classification}
\makeatother



\begin{document}

\begin{abstract}
	We investigate different approaches to transform a given binomial
	into a monomial via blowing up appropriate centers. 
	In particular, we develop explicit implementations in {\sc Singular},
	which allow to make a comparison on the basis of numerous examples.
	We focus on a local variant, where centers are not required to be chosen globally.
	Moreover, we do not necessarily demand that centers are contained in the singular locus. 
	Despite these restrictions, the techniques are connected to the computation of $ p $-adic integral whose data is given by finitely many binomials. 
\end{abstract}

\maketitle

\section{Introduction}
	The goal of this article is to investigate and to compare different methods to transform
	a binomial into a monomial
	via blowing up appropriate centers. 
	Within this, we develop explicit implementations so that the different approaches can be compared on the basis of numerous examples.
	This is motivated by the DFG-project ``Order zeta functions and resolutions of singularities" (principal investigators: Christopher Voll and Anne Fr\"uhbis-Kr\"uger), 
	of which the second named author is part of. 
	In there, a technique for the explicit computation of special 
	$ p $-adic integrals is developed using monomialization. 
	In particular, the structure of the integrals considered there allow a reduction to the case of finitely many binomials. 
	An increasing complexity in the $ p $-adic integrals 
	(which is reflected in a rapidly increasing number of variables and binomials)
	requires to find monomialization algorithms which keep the numbers of blowups and of final charts that have to be considered small.
	Since the problem is of combinatorial nature, we focus on the situation over a field, while we briefly discuss the case over $ \IZ_p $ in Remark~\ref{Rk:ZZ_p}.  
	
	\smallskip 
	
	Let $ K $ be a field and let 
	$ f = \ux^A - \cc \ux^B \in K[\ux] : = K [x_1, \ldots, x_n]  $ be a binomial,
	where $ \cc \in K^\times $ and 
	$ \ux^{A} = x_1^{A_{1}} \cdots x_n^{A_{n}} $ for 
	$ A = (A_{1}, \ldots, A_{n}) \in \IZ_{\ge 0}^n $. 
	We say that $ f $ is {\em locally a monomial} if 
	for every point $ q \in\IA_K^n = \Spec(K[\ux]) $ there exists a regular system of parameters for the local ring $ \mathcal{O}_{\IA_K^n,q} $ such that $ f $ is a monomial times a unit with respect to these parameters. 
	For example, $  x_1^3 x_2^2 (1 - x_1)^4 $ is locally monomial since for every $ q \in \IA_K^2 $ at least $ x_1 $ or $ 1-x_1 $ is a unit.
	On the other hand,
	$ x_1 x_2 (x_1+x_2) $ is not
	locally a monomial as there is no regular system of parameter for the local ring at the origin such that 
		$ x_1 x_2 (x_1+x_2) $ becomes a monomial times a unit.
	The tool that we want to apply to make a binomial locally monomial are blowups,
	e.g., $  x_1 x_2 (x_1+x_2) $ becomes locally monomial after blowing up with center $ V (x_1, x_2) $.
	
	The blowup $ \pi \colon \Bl_D (\IA_K^n) \to \IA_K^n  $ in a regular center $ D = V (x_i \mid i \in I ) $, for some $ I \subseteq \{ 1, \ldots, n \} $ is covered by the standard charts $ U_i := D_+(X_i) \cong \IA_K^n $, $ i\in I $, where $D_+(X_i) :=\Bl_D (\IA_K^n)  \setminus V(X_i) $ is the complement of $V(X_i)$.
	On $ U_i $, the blowup $ \pi $ is given by the morphism
	\[
	\begin{array}{rcll}
		K[x_1,\ldots, x_n] & \longrightarrow & K[x_1',\ldots, x_n'] =: K[\ux']
		\\[5pt]
		x_j & \mapsto & x_i' \, x_j', & \mbox{if } j \in I \setminus \{ i \}, 
		\\
		x_j & \mapsto & x_j', & \mbox{if } j = i \mbox{ or } j \notin I,
		\\ 
		\lambda & \mapsto & \lambda, & \mbox{for } \lambda \in K.  
	\end{array} 
	\] 
	The image of $ f \in K[\ux] $ in $ K[\ux'] $ is called the total transform of $ f $ in $ U_i $. 
	On the level of exponents, the above morphism corresponds to the map 
	\[  
		\phi_{\pi,i} \colon \IZ_{\ge 0}^n \longrightarrow  \IZ_{\ge 0}^n ,
	\]
	where $	A = (A_1, \ldots, A_n) $ is mapped to
	$ A' = (A_1', \ldots, A_n') $,
	which is defined by
	$ A_i' := \sum_{j\in I} A_j $
	and $ A_j' := A_j $ if $ j \ne i $.
	For the later use, recall that $ |A| := \sum_{i=1}^n A_i $. 
	
	In every $ U_i \cong \IA_K^n $, we may choose a center $ D_i $ of the same shape as $ D $ and we can iterate this to obtain a sequence of local blowups.
	We say that a finite sequence of local blowups obtained by iterating the previous procedure is a {\em local monomialization of $ f = \ux^{A} - \cc \ux^{B} $} if
	the total transform of $ f $ is locally a monomial in every final chart of the blowup tree. 
	For example, the latter is the case if the total transform is of the form (up to multiplication by a non-zero constant) 
	\begin{equation}
	\label{eq:ex_loc_mon_cond}
		\ux^C (1 - \mu \ux^{B'}) 
		\ \ 
		\mbox{ or }
		\ \  
		\ux^C (x_i - \mu \ux^{B'}),
	\end{equation}	
	for $ C,B' \in \IZ_{\ge 0}^n $ and $ \mu \in K^\times $, 
	where we require in the second case that $C_i = B'_i = 0 $ for the special $ i \in \{ 1, \ldots, n \} $ given. 	
	{The last hypothesis implies that we may introduce the coordinate change $ y_i := x_i - \mu \ux^{B'} $ such that $ \ux^C (x_i - \mu \ux^{B'}) = \ux^C y_i $ becomes a monomial.
	It is clear that $ \ux^C (1 - \mu \ux^{B'}) $ is a monomial times a unit if $ 1 - \mu \ux^{B'} $ is invertible in $ \mathcal{O}_{\IA_K^n,q} $.
	On the other hand, if $ 1 - \mu \ux^{B'} $ is not invertible, we have to distinguish two cases. 
	Let $ d $ be the greatest common divisor of the entries of $ B' $. 
	First, if $ d = 1 $ or if $ d > 1 $ and $ \mu $ has no $ d' $-th root in $ K $ with $ d' | d $, then $ 1 - \mu \ux^{B'} $ is irreducible and regular.
	Hence, it can be extended to a regular system of parameters at $ q $. 
	Otherwise, $ 1 - \mu \ux^{B'} $ is not irreducible, but then all but one of the distinct irreducible factors are invertible since they arise from the factorization of $ T^{d'} - 1 \in K[T] $ (with maximal $ d' $ as in the first case and $ T $ a substitute for $ ( \mu \ux^{B'} )^{1/d'} \in K[\ux] $).
	Note that in both cases all $ x_i $ with $ B_i' \neq 0 $ are units in $ \mathcal{O}_{\IA_K^n,q} $}

	Observe that \eqref{eq:ex_loc_mon_cond} is not equivalent to being locally monomial. 
	For example, $ x^2 - y^2 \in \mathbb{C}[x,y] $ and $ x^p + y^p \in \mathbb{F}_p[x,y] $ ($ p \in \IZ $ prime) do not fulfill  \eqref{eq:ex_loc_mon_cond},
	but they are monomial after a suitable change of variables
	$ \widetilde x := x+ y, \widetilde y := x-y $ in the first example
	and $ \bar x := x + y $ for the second.

	We choose the centers in each chart independent of the other charts,
	i.e., we do not necessarily obtain a sequence of global blowups. 
	This provides more freedom in the choice of the center and is still sufficient for the explicit computations,
	where the local charts are interpreted as case distinctions. 
	In \cite[Question~5.6]{Bernd}, this local variant of monomialization via blowups is discussed in the context of resolution of singularities.

	\smallskip
	
	The differences in the methods appear in choice of the center for the next blowup. 
	Let us briefly explain the variants.
	Consider a binomial 
	\[  
		f = \ux^C (\ux^A - \cc \ux^B) , 
	\]
	where $ A, B, C \in \IZ_{\ge 0}^n $ are such that $ A_i B_i = 0 $ for all $ i \in \{ 1, \ldots, n \} $.
	Set 
	\[
		g : = \ux^A - \cc \ux^B .
	\]
	Suppose that $ f $ is not locally monomial. 
	The basic idea for the four variants that we consider are:
	\begin{enumerate}
		\item[\bf (1)] 
		{\em Centers contained in the locus of maximal order}
		(section~\ref{theory:maxord} {/ Construction~\ref{Constr:center_max_ord}}).
		Choose $ D = V (x_i \mid i \in I ) $ such that $ D $ is contained in the locus of maximal order of $ V(g) $.
		This is equivalent to imposing 
		\[  
			\min \big\{\sum_{i\in I} A_i, \sum_{i\in I} B_i \big\} =
			\min \big\{ |A|, |B| \big\} .
		\] 
		
		\item[\bf (2)] 
		{\em Centers of codimension two}
		(section~\ref{theory:codim2} {/ Construction~\ref{Constr:center_codim2}}).
		Choose $ i , j \in \{ 1, \ldots, n \} $ such that $ A_i \neq 0 $, $ B_j \neq 0 $ and both are maximal.
		Then, the center for the blowup is 
		 $ D = V (x_i, x_j ) $.

		\smallskip 

		\item[\bf (3)] 
		{\em Centers of minimal codimension contained in the singular locus}
		(section~\ref{theory:singloc}~{/ Construction~\ref{Constr:center_non-monomial}}).
		If $ \min \big\{ |A|, |B| \big\} \ge 2 $,
		choose $ I \subseteq \{ 1, \ldots, n \} $ such that $ \sum_{i\in I} A_i \geq 2 $, $ \sum_{i\in I} B_i \geq 2 $
		and such that $ \#I $ is minimal with this property.
		Then $ D := V(x_i \mid i \in I ) $.
		Else, choose $ D $ as in (2).
		
		\smallskip

		\item[\bf (4)] 
		{\em Centers of minimal codimension contained in an exceptional divisor or contained in the singular locus}
		(section~\ref{theory:preimagesingloc} {/ Construction~\ref{Constr:center_exceptional_non-monomial}}).
		If there is a center $ D $ of type (2) contained in an exceptional divisor, we choose this.
		Otherwise, we follow (3).
	\end{enumerate}
	Note that the centers are not necessarily uniquely determined and one might have to make a choice. 
	In the respective sections, we provide examples for this phenomenon.

	While (1) follows the usual approach to resolution of singularities,
	method (2) solely has the motivation to minimize the numbers of charts after a single blowup in order to make it easier to control the transform of the binomial. 
	In particular, the resulting morphism is not necessarily an isomorphism outside of the singular locus of $ g $.
	In (3), we consider a mixture of (1) and (2);
	we try to choose the centers as large as possible, but moreover, we require that
	the centers are contained in the singular locus of $ g $ (resp.~its variant after a blowup) and if the latter is empty, we follow (2).
	Finally, in (4), we relax the last condition (3) and 
	allow centers of codimension two,
	which are not necessarily contained in the singular locus of $ g $, 
	if they are contained in an exceptional divisor.

	In the respective sections, we discuss the benefits of each approach
	and show the termination of the local monomialization algorithm resulting from the different choices, see Propositions~\ref{Prop:max_ord_terminates}, \ref{Prop:codim2_terminates}, \ref{Prop:singlocus_terminates}, and Corollary~\ref{Cor:except_singlocus_terminates}, respectively.
	Along this, we discuss algorithms for explicit implementations of each variant to monomialize a binomial, which have been realized in the open source computer algebra system {\sc Singular} \cite{Sin}.	
			
	We study the binomial through the appearing exponents $ A,B,C $ and their behavior along the blowups using $ \phi_{\pi,i} $. 
	More precisely, we deduce from the exponents numerical measures that detect how far the given binomial is from being monomial. 
	Then we show that the respective measure decreases strictly after a single blowup following the corresponding procedure and that a strict decrease may only appear finitely many times. 
	
	In the final section, we analyze the different variants for the choice of the centers by comparing the numbers of charts for a worst case scenario
	and for numerous explicit examples.
	The latter is based on an implementation of the discussed algorithms in {\sc Singular}. 
	As a measure for the complexity, we consider the number of charts along the monomialization process as well as at the end.
	A brief summary is that variant (1) has a significant larger number than the other three, while (2) is  often the most efficient algorithm. 
	But there exist cases, where (3) and (4) are slightly better than (2). 
	As mentioned above, we have to make a choice among the possible centers. 
	For some of the example, we study the different results if we vary the choices.

	\smallskip 

	\begin{acknowledgements}
		We thank Anne Fr\"uhbis-Kr\"uger for discussions, comments, and her guidance with {\sc Singular}.
		Further, we thank the referee for their helpful comments.
	\end{acknowledgements}

\smallskip

\section{The basic algorithmic framework and blowups}

We begin by discussing the basic structure for the implementation of a monomialization procedure. 
Within this, we also introduce numerical invariants, which we later use to prove the termination of the different monomialization methods. 
Furthermore, we provide an algorithm testing whether a given binomial 
fulfills condition \eqref{eq:ex_loc_mon_cond}, which implies that the binomial 
is locally monomial. 
Finally, we give an implementation of the transformation of the exponents along a blowup.

\smallskip 

The main method is the same for all of the four strategies.
The difference of the monomialization methods appears only in the choice of the center.
In Figure~\ref{flowchart}, we provide the flow chart of the main method and in Algorithm~\ref{main} the precise implementation. 

\begin{figure}
	\begin{tikzpicture}[node distance = 2cm,auto]
	\node [cloud] at (0,0) (start) {binomial $f=\ux^{\tilde{A}}- \cc \ux^{\tilde{B}}$};
	\node [block] at (0,-2) (init) {initialization $f=\ux^{C}(\ux^{A} -\cc \ux^{B})$};
	
	\node [decision] at (0,-5) (test){finished?};
	
	\node [cloud] at (4,-5)(return){return};
	
	\node [block] at (0,-8) (center) {center};
	
	\node [block] at (0,-10) (trans){transformation};
	
	
	\draw[->] (start)  --  (init);
	\draw[->] (init)  --  (test);
	\draw[->] (center)  --  (trans);
	\draw[->] (trans) -- (-3,-10) -- (-3,-5) -- (test);
	\draw[->] (test)  -- node[above]{yes}  (return);
	\draw[->] (test)  -- node[right]{no}   (center);
	\end{tikzpicture}
	\caption{Flow chart of the main method.}
	\label{flowchart}
\end{figure}
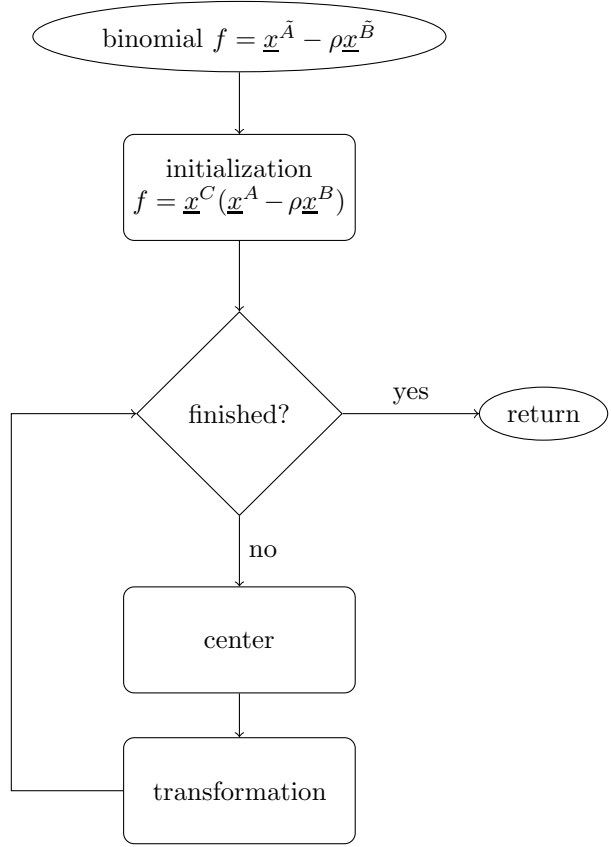

The implementations are of combinatorial nature.
Instead of working with the binomial $ \ux^C (\ux^A - \cc \ux^B ) $, 
we consider the exponents $ A = ( A_1,\ldots,A_n) $, $ B = ( B_1,\ldots,B_n)$ and $ C = (C_1,\ldots,C_n) $.
Additionally, we introduce a vector of ones and zeros $ E = ( E_1,\ldots, E_n) \in \{ 0, 1 \}^n $, 
where we encode, which variables correspond to exceptional divisors,
i.e., $ E_i = 1 $ if and only if $ \operatorname{div}(x_i) $ is an exceptional divisor.  
This will be necessary for the variant for choosing the center of section~\ref{theory:preimagesingloc}, {see Construction~\ref{Constr:center_exceptional_non-monomial}}.

Before stating and explaining Algorithm~\ref{main}, 
let us introduce the following numbers,
which play an important role in parts of the monomialization procedures discussed in the present work.

\begin{defi}
	\label{Def:Invariant_codim_2}
	Let $ K $ be any field.
	Let $ g = \ux^A - \cc \ux^B \in K[\ux] $,
	for $ \cc \in K^\times $ and $ A, B \in \IZ_{\ge 0}^n $ 
	such that $ A_i B_i = 0 $ for all $ i \in \{ 1, \ldots, n \} $.
	We define
	\begin{eqnarray*}
		\alpha(g)&:=& \max\{A_i \mid i \in \{1,\dots,n\}\} 
		\\
		\mathfrak{a}(g)&:=& \#\{i \in \{1,\dots,n\} \mid A_i = \alpha(g) \} 
		\\
		\beta(g)&:=& \max\{B_i \mid i \in \{1,\dots,n\}\} 
		\\
		\mathfrak{b}(g)&:=& \#\{i \in \{1,\dots,n\} \mid B_i = \beta(g)\}
		\\
		\iota(g) &:=& \left( \alpha(g), \mathfrak{a}(g), \beta(g), \mathfrak{b}(g)\right) \in \mathbb{Z}_{\ge 0}^4.
	\end{eqnarray*}
	Here, we equip $ \IZ_{\ge 0}^4 $ with the lexicographical ordering $ \leq_\lex $.
	Given $f = \ux^{C} g $ with $ C \in \IZ_{\ge 0}^n $,
	we also write $ \alpha(f) := \alpha(g), \ldots,  \iota(f) := \iota(g) $,
	if no confusion is possible.
\end{defi}

Clearly, $ \iota(g) $ depends on the order of the monomials in $ g $.
In general, we have $ \iota(\ux^A - \ux^B) \neq \iota(-\ux^B + \ux^A) $.
Since we fix an order of the monomials in an implementation anyways, 
we will work later with the string $ (A,B) $ instead of $ g $
and
we neglect the matter of making $ \iota(g) $ independent of the order of the monomials.

\begin{algorithm}
	\caption{Main method (for a description see Remark~\ref{Rk:Algo1})}
	\begin{algorithmic}[1]
		\INPUT $ f = \ux^{\tilde{A}}- \cc \ux^{\tilde{B}} $, $ \mbox{mode} \in \{1,2,3,4\}$,
		where $ (\ux) = (x_1, \ldots, x_n) $, $ \tilde{A}, \tilde{B} \in \IZ_{\ge 0}^n $
		\OUTPUT 
		list $ L $, where $ L[i] $ is the data of the $ i $-th chart

		\State list $Lf $
		\State $ Lf[1] $ = list$(A,B,C,(\underline{0}))$,
		where $f=\ux^{C}(\ux^{A} -\cc \ux^{B})$ such that $ A_i B_i  = 0 $, for all $ i $,
		and $ ( \underline{0}) \in \IZ_{\ge 0}^n $

		\If{check\_finished($Lf[1]$)}
		\ \Comment{see Algorithm~\ref{alg:locallymonomial}}
		\State $ Lf[2] = Lf[3] = Lf[4] = \varnothing $
		\State $ L[1] = Lf $
		\State\Return $ L $
		\EndIf
		
		\State $ Lf[2] $ = $ \iota(f)$ 
		= list$( \alpha, \mathfrak{a}, \beta, \mathfrak{b} ) \in \mathbb{Z}_{\ge 0}^4 $ 
		\ \Comment{Definition~\ref{Def:Invariant_codim_2}}
		\State  $ I $\_{center} = compute\_center$ ( Lf ,\mbox{mode}) $
		
		\State $ Lf[3] $ = $ I $\_{center} 
		\State	intmat path$[2][1] = 0,-1$
		\State	$ Lf[4] $ = path
		\State $ L [1]  = Lf $

		\State successors = fill\_list\_for\_next\_charts$(L[1],1) $
		\For{$ L_+ \in $ successors}
		\State $ L[size(L)+1] = L_+ $
		\EndFor

		\State $i = 2$
		\While{ $ i \le size(L)$}
		\State $L[i]$ = transformation($L[i] $, $ \mbox{mode} $)
		\ \Comment{see Algorithm~\ref{transformation}}
		\If{check\_finished($ L[i][1] $) == false}
		\State successors = fill\_list\_for\_next\_charts$(L[i],i) $
		\For{$ L_+ \in $ successors}
		\State $ L[size(L)+1] = L_+ $
		\EndFor
		\EndIf
		\State $i = i+1$
		\EndWhile
		\State\Return L
	\end{algorithmic}
	\label{main}
\end{algorithm}

\begin{Rk}[Algorithm~\ref{main}]
	\label{Rk:Algo1}
	The {\em input}  
	is:
	\begin{itemize}
		\item 
		a binomial
		$ f = \ux^{\tilde{A}}- \cc \ux^{\tilde{B}} \in K [\ux] = K[x_1, \ldots, x_n]$;
		
		\item 
		an integer $ \mbox{mode} \in \{ 1, 2, 3, 4 \} $,
		which determines the method for choosing the centers;
	\end{itemize}  
	The {\em output} of Algorithm~\ref{main}
	is a list $ L $ 
	which consists of all charts of the monomialization process.
	From this list, one can determine a list of all leaves of the monomialization tree,
	i.e., of all final charts.
	The {\em data in a chart $ L[i] $} is of the following form:
	
	\smallskip 
	
	\begin{itemize} 
		\item[{$[1]$}] $(A,B,C,E) \in ( \IZ_{\ge 0}^n )^4 $ such that 
		the total transform of $ f $ in the chart is 
		$ f = \ux^{C}(\ux^{A} -\cc \ux^{B}) $ 
		and $ A_i B_i  = 0 $, for all $ i \in \{ 1, \ldots, n \} $.
		Furthermore,  $ E \in \{ 0, 1 \}^n $ is the vector encoding which variables are exceptional.

		\smallskip

		\item[{$[2]$}] $ \iota(f) = ( \alpha, \mathfrak{a}, \beta, \mathfrak{b} ) \in \mathbb{Z}_{\ge 0}^4 $ is the measure introduced in Definition~\ref{Def:Invariant_codim_2}.

		\smallskip 
		
		\item[{$[3]$}] $  I\mbox{\_center} \subseteq \{ 1, \ldots, n \} $
		is the index set such that 
		$ \langle x_i \mid i \in I\mbox{\_center} \rangle $ is the ideal defining the center for the next blowup.

		\smallskip 
		
		\item[{$[4]$}] a pathmatrix
		$\begin{pmatrix}
		0 & \cdots & x\\
		-1 & \cdots & y
		\end{pmatrix}
		$ such that $x$ is the number of the predecessor chart. 
		The successors of the predecessors are labeled from 1 to $\#$successors.
		The number $ y $ indicates which of these successors the given chart is. 
		This entry is not important for the monomialization process, but for a later evaluation of the final data to keep track of the global picture.
	\end{itemize}

	\smallskip 
	
	First, Algorithm~\ref{main} performs an initialization,
	by determining the exponents such that 
	$  f = \ux^{\tilde{A}}- \cc \ux^{\tilde{B}}  =  \ux^{C}(\ux^{A} -\cc \ux^{B}) $
	has the desired form. 
	Since there are no exceptional divisors yet $ E = ( \underline{0}) $. 
	Then, we check whether $ f $ 
	verifies condition~\eqref{eq:ex_loc_mon_cond}
	using the method {\em check\_finished},
	see Algorithm~\ref{alg:locallymonomial}.
	If \eqref{eq:ex_loc_mon_cond} holds, then the binomial is locally monomial and we fill the list $ Lf $ with trivial data and return $ L $. 
	
	If \eqref{eq:ex_loc_mon_cond} is not fulfilled,
	we determine the full data of the chart (lines 7--11).
	In there, \textit{compute\_center$ ( Lf ,\mbox{mode}) $} is the method determining the index set of the center for the next blowup.
	The input $ \mbox{mode} \in \{ 1, 2, 3, 4 \} $ 	fixes, which of our four methods is used. 
	In the following sections, we discuss the methods for choosing the center in details.
	Furthermore, {\em intmat} initiates an integer matrix called {\em path}, which encodes the tree structure of the monomialization process.
	
	Then, the method 
	\textit{fill\_list\_for\_next\_charts} copies the data from $L[i]$ (in line 13 for $ i = 1 $) to a list 
	\textit{successors}, 
	which contains as many charts as needed (depending on the 
	center of the upcoming blowup). 
	The only difference in the data of the charts 
	in \textit{successors} is the adapted \textit{path} matrix which contains the 
	tree of the blowup procedure.
	After this, the charts from list \textit{successors} are added to the end of $ L $.
	Let us explain this step more in details: 
		If we want to determine the successors for $ L[i] $
		and if the upcoming blowup has $ m $ many charts 
		and if $ L $ has $ k $ entries in total (with $ k \ge i $) before adding the successors to it, 
		then the successors become the entries $ L[k+1], \ldots, L[k+m] $
		and 
		we extend {\em path} for the successor $ L[k+j] $ by the column $ (i, j)^T $ at the end since it is the $ j $-th chart of the blowup in $ L[i] $, where $ j \in \{ 1, \ldots, m \} $.

	In the while-loop (lines 17--23), 
	the data of
	$L[i]$ 
	(except for {\em path})
	is modified 
	such that it becomes the transformed version of its predecessor with respect to the previously determined center.
	The transformation algorithm is provided in Algorithm~\ref{transformation} and described in Observation~\ref{Obs:bu}.
	Finally, we verify whether the data in $L[i]$ fulfills \eqref{eq:ex_loc_mon_cond}. 
	If so, then the procedure continues 
	with the entry $ L[i+1] $ if it exists
	(i.e., with the next chart which needs to be handled) 
	or it stops if $ L $ has $ i $ entries.
	Otherwise, if \eqref{eq:ex_loc_mon_cond} does not hold, we blow up, the successor charts are stored at the end of 
	the list, analogous to before,
	and we continue with the chart $ i + 1 $.
	The while-loop will eventually end since we will show in the following sections that the respective monomialization procedures terminate. 
	
	\hfill {\em (End of Remark~\ref{Rk:Algo1}.)}
\end{Rk}

\begin{algorithm}
	\caption{check\_finished}
	\begin{algorithmic}[1]
		\INPUT list $ ( A, B, C, E ) $  of vectors in $ \IZ_{\ge 0}^n $ such that $ A \neq B $ and $ A_i B_i = 0 $ for all $ i \in \{ 1, \ldots, n \} $
		\OUTPUT true, if the binomial $ x^C (x^A - x^B ) $ fulfills \eqref{eq:ex_loc_mon_cond}; false otherwise
		\If{$\min \{|A|,|B|\}==0$}
		\State\Return true
		\EndIf
		\If{$|A| = 1 $ or $|B| = 1$}
		\If{$\exists \, i: C_i = 0 $ and $( A_i = 1 $ and $|A|=1)$ or $( B_i = 1 $ and $|B|=1)$ }
		\State\Return true
		\EndIf
		\EndIf
		
		\State\Return false
	\end{algorithmic}
	\label{alg:locallymonomial}
\end{algorithm}

Note that in Algorithm~\ref{alg:locallymonomial} it does not matter that the coefficient is $ 1  $ instead of $ \rho $.
The result is the same.

\begin{Obs}
	\label{Obs:bu} 
	Let $ f  = \ux^C (\ux^A - \cc \ux^B) \in K[\ux] $ be a binomial with $ A_i B_i = 0 $ for all $ i \in \{ 1, \ldots, n \} $.
	Let us consider how the exponents change along the blowup with center $ D = V ( x_1, \ldots, x_m ) $, for some $ m \in \{ 2 , \ldots, n \} $. 
	In the $ X_1 $-chart, we have
	\[
	x_1 = x_1' , 
	\ x_2 = x_1' x_2' , 
	\ \ldots, 
	\ x_m = x_1' x_m' ,
	\ x_{m+1} = x_{m+1}',
	\ \ldots,  
	\ x_{n} = x_{n}'.
	\]
	On the level of exponents, this provides
	\[
	A = (A_1, A_2, \ldots, A_n) \mapsto 
	\widetilde A' := 
	(A_1+A_2+\ldots + A_m , A_2, \ldots, A_n)
	\]
	and analogous for $ B $ and $ C $. 
	
	We factor the total transform of $ f $ as $ \ux'^{C'} (\ux'^{A'} - \cc \ux'^{B'}) $ such that $ A_i' B_i' = 0 $ for all $ i \in \{ 1, \ldots, n \} $. 
	If we set 
	\[  
	\delta := \min \{ A_1 + \ldots + A_m,  B_1 + \ldots + B_m \} ,
	\] 
	then we get 
	\[ 
	\begin{array}{l} 
	A' = (A_1+A_2+\ldots + A_m - \delta , A_2, \ldots, A_n), \\[5pt]
	B' = (B_1+B_2+\ldots + B_m - \delta , B_2, \ldots, B_n), \\[5pt]
	C' = (C_1+C_2+\ldots + C_m + \delta , C_2, \ldots, C_n).
	\end{array}
	\]
	The other charts are analogous. 
	Furthermore, it is straight forward to adapt this to blowups in centers of the form $ V (x_i \mid i \in I ) $, where $ I \subseteq \{ 1, \ldots, n \} $ is not necessarily equal to $ \{ 1, \ldots, m \} $. 
\end{Obs}

This leads to Algorithm~\ref{transformation} for determining the transform of a binomial in a given chart of the blowup in $ V (x_i \mid i \in I ) $.

\begin{algorithm}
	\caption{transformation (for a description see Remark~\ref{Rk:Algo_transfo})}
	\begin{algorithmic}[1]
		\INPUT list $ M $, $ \mbox{mode} \in \{1,2,3,4\}$,
		where $ M $ is of the same form as $ L[i] $ in Remark~\ref{Rk:Algo1}
		\OUTPUT list $retList$ which is the transformed variant of chart $ M $
		
		\State $ (A,B,C,E) = M[1] $
		\State path = $ M[4] $
		\State $ I = M[3] $  \Comment {the index set of the center}
		\State $ i $ = path[2,ncols(path)]  \Comment {so $ M $ is the $ X_i$-chart of the blowup} \Statex\Comment{with center $ V (x_i\mid i \in I) $}
		\State $\delta =  \min\{ \sum_{j \in I} A_j, \sum_{j \in I} B_j \}$
		\State $A_i = \sum_{j \in I} A_j \, -  \delta$
		\State $B_i = \sum_{j \in I} B_j \, - \delta $
		\State $C_i = \sum_{j\in I } C_j \, + \delta$
		\State $E_i = 1  $
		\State $ retList[1] $ = list$ (A,B,C,E) $
		\If{check\_finished($retList[1]$)}
		\State $ retList[2] = retList[3] = \varnothing $
		\State $ retList[4]$  = path;
		\State\Return $ retList $
		\EndIf
		\State $ retList[2] = \iota(\ux^{C}(\ux^{A} -\cc \ux^{B}))$ 
		\Comment  {Definition~\ref{Def:Invariant_codim_2}}
		\State	$ I $\_center = compute\_center($retList[1], \mbox{mode} $)
		\State $ retList[3] = I $\_center
		\State $ retList[4] = $ path
		\State\Return $ retList $
	\end{algorithmic}
	\label{transformation}
\end{algorithm}

\begin{Rk}[Algorithm~\ref{transformation}] 
	\label{Rk:Algo_transfo}
	The {\em input}  
	is:
	\begin{itemize}
		\item 
		a list $ M $,
		which represents the data of a chart
		and hence is of the same form as $ L[i] $ in 
		Remark~\ref{Rk:Algo1}
		
		\item 
		an integer $ \mbox{mode} \in \{ 1, 2, 3, 4 \} $,
		which determines the method for choosing the centers;
	\end{itemize}  
	The {\em output} of Algorithm~\ref{transformation}
	is the transformed data of the input chart.
	
	First, we initialize the data (lines 1--4).
	In particular, we specify the index set $ I $ corresponding to the center of the blowup
	and the element $ i \in I $ such that $ M $ corresponds to the $ X_i $-chart of the blowup.
	After that, we transform the exponents $ (A,B,C,E) $
	as described in Observation~\ref{Obs:bu}
	(but now for the general case) and mark the variable $ x_i $ as exceptional in lines 6--9. 
	Finally, we check whether the transformed binomial fulfills \eqref{eq:ex_loc_mon_cond} 
	and determine the remaining data so that the output data is of the same form as $ L[i] $ in Remark~\ref{Rk:Algo1}. 
	(Note that the path matrix is extended in line 20 of Algorithm~\ref{main}.)
\end{Rk}

The only part of implementation which differs in the various modes is the computation of the center. 
We have seen above that every other method of the 
implementation only uses the \textit{mode} parameter in order to call the 
\textit{compute\_center}-method which is described later.

\smallskip

\section{Centers contained in the locus of maximal order}
\label{theory:maxord}

In this section we discuss the first of the four variants for the choice of center in details.
We fix a binomial 
\[  
f = \ux^C (\ux^A - \cc \ux^B) \in K[\ux] = K[x_1, \ldots, x_n], 
\]
where $ \cc \in K^\times $, $ A, B, C \in \IZ_{\ge 0}^n $ are such that $ A_i B_i = 0 $ for all $ i \in \{ 1, \ldots, n \} $
and $ K $ is a field.
We set
\[ 
	g : = \ux^A - \cc \ux^B .
\]  
Observe that the condition $ A_i B_i = 0 $ implies that at least one of them is zero and hence $ x_i $ cannot be factored from $ g $.

If $ g = 1 - \mu \ux^B $ or $ {f}  = \ux^C (x_i - \mu \ux^B) $, for $ C, B \in \IZ_{\ge 0}^n $ and $ \mu \in K^\times $, 
where we require in the second case that $C_i = 0 $ for the given $ i $, then $ f $ is locally monomial and no blowups are required. 
Hence, throughout this section, we assume  
that \eqref{eq:ex_loc_mon_cond} is not fulfilled, i.e., that the following condition holds:
\begin{equation}
\label{eq:cond_not_loc_mon}
\left\{ 
\hspace{-10pt}
\begin{minipage}{0.8\textwidth}
	\begin{itemize}
	\item either $ \min \{ |A|,|B|\} \geq 2  $, or
	
	\smallskip 
	
	\item  $ \min \{ |A|,|B|\} = 1 $ and $ C_i \neq 0 $ for every $ i $ such that $ A_i = 1 $ if $ |A|=1 $, or $ B_i = 1 $ if $ |B|=1 $.
	
	\end{itemize} 
\end{minipage}
\right.
\end{equation}
Again, let us point out that \eqref{eq:cond_not_loc_mon} does not imply that $ f $ is not locally monomial. 
For example, $ x^2 - y^2 \in \mathbb{C}[x,y] $ is monomial after introducing $ \widetilde x := x+y $, $ \widetilde y := x-y $, but \eqref{eq:cond_not_loc_mon} holds.
Nonetheless, it is not hard to test with a computer whether \eqref{eq:ex_loc_mon_cond} is true
and thus we admit that we might perform some blowups, which are not needed.

A common approach in resolution of singularities for a hypersurface $ V(g) $ 
is to consider regular centers contained in its locus of maximal order. 
First, let us recall its definition. 

\begin{defi}
	Let $ h \in K[\ux] \setminus \{ 0 \}  $ be a non-zero polynomial and set $ X := V(h) \subset \Spec(K[\ux]) = \IA_K^n $. 
	Let $ \mathfrak p \subset K [\ux] $ be a prime ideal. 
	Geometrically, we denote by $ x_{\mathfrak{p}} \in  \IA_K^n $ the point that corresponds to $ \mathfrak p $.
	\begin{enumerate}
		\item  
		The {\em order of $ X $ at $ x_{\mathfrak p} $ (resp.~of $ h $ at $ \mathfrak{p} $)} is defined as 
	\begin{equation*}
	\ord_{x_{\mathfrak p}} (X) := 
	\ord_{\mathfrak{p}}(h):= \sup\{ \ell \in \IZ_{\ge 0} \mid h \in \mathfrak{p}^\ell \}. 
	\end{equation*}
	
	\item 
	The {\em maximal order of $ X $} 
	is defined as
	\begin{equation*}
	\maxord (X)
	:= \sup\{  \ord_{x_{\mathfrak p}} (X) \mid x_{\mathfrak p}  \in X \}
	\end{equation*}
	and the {\em locus of maximal order of $ X $} is
	\begin{equation*}
	\Maxord(X) := 
	\{ x_{\mathfrak p}  \in X 
	\mid 
	\ord_{x_{\mathfrak p}} (X) = \maxord(X) \}.
	\end{equation*}
	\end{enumerate}
\end{defi}

Sometimes, we use the notation $ \maxord(h) := \maxord(X) $ or $ \Maxord(h) := \Maxord (X) $. 
Since the order is upper semi-continuous (see \cite[Chapter III, \S 3, Corollary 1, p.~220]{Hiro64}), the level sets 
$ \{ x_p \in X \mid  \ord_{x_{\mathfrak p}}(X) \geq a \} $, $ a \in \IZ_{\ge 0} $, are Zariski closed. 
In particular, this is true for $ \Maxord(X) $.

If $ X $ is a variety, which is not a hypersurface, then 
the order is not an appropriate measure for the complexity of the singularity, see \cite[Example~2.7]{FRS}.

\smallskip 


\begin{Lem}
	\label{Lem:Center_in_max_ord}
	Let $ g = \ux^A - \cc \ux^B \in K[\ux] $ be a binomial {such that $ \min \{ |A|,|B|\} \geq  1 $ and $ A_i B_i = 0 $ for all $ i \in \{ 1, \ldots, n \} $.} 
	{Let} $ I \subseteq \{ 1, \ldots, n \} $ be any subset. 
	Set $ X := V(g) $ 
	and $ D_I := V(x_i \mid i \in I ) $.  
	We have:
	\begin{enumerate}
		\item 
		$ \maxord(X) = \min\{ |A|, |B|\} $.
		
		\item 
		$ D_I \subseteq\Maxord(X) \Longleftrightarrow \min \big\{\sum_{i\in I} A_i, \sum_{i\in I} B_i \big\} =
		\min \big\{ |A|, |B| \big\}  $.  
	\end{enumerate}
\end{Lem}

{The condition $ \min \{ |A|,|B|\} \geq  1 $ comes from the fact that we assume \eqref{eq:cond_not_loc_mon} to hold. 
	Notice that it is necessary, e.g., for $ g = x_1 x_2 - 1 $ the maximal order is $ 1 $ and not zero, which can be seen by computing the order at $ \langle x_1 - 1, x_2 - 1 \rangle $,
	cf.~Example~\ref{Ex:there_is_more_for_m>1}.}

\begin{proof}[Proof of Lemma~\ref{Lem:Center_in_max_ord}]
	{Let $ \mathfrak{m} := \langle x_1, \ldots, x_n \rangle $ be the maximal ideal corresponding to the origin.
	We have 
	$ \min\{ |A|, |B|\} =  \ord_{\mathfrak{m}} (g) \leq \maxord(X)  $.
	
	Suppose there is some prime ideal $ \mathfrak{p} \subset K[\ux] $ with $  \ord_{\mathfrak{p}} (g) >  \ord_{\mathfrak{m}} (g) $.
	This implies, 
	if we base change to an algebraic closure $ \overline{K} $ of $ K $, 
	then there is a maximal ideal $ \mathfrak{n} \subset \overline{K}[\ux] $ such that 
	$ \ord_{\mathfrak{n}}(g) > \ord_{\mathfrak{m}}(g) = \min\{ |A|, |B|\} $. 
	Since $ \overline{K} $ is algebraically closed and $ \mathfrak{n} $ is a maximal ideal, 
	there are $ c_1, \ldots, c_n \in \overline{K} $ such that $ \mathfrak{n} = \langle x_1 - c_1, \ldots, x_n - c_n  \rangle $ by Hilbert's Nullstellensatz.
	  
	Set $ I_1 := \{ i \in \{ 1, \ldots, n \} \mid A_i \neq 0 \} $ and $ I_2 := \{ i \in \{ 1, \ldots, n \} \mid B_i \neq 0 \} $ . 
	Since $ A_i B_i = 0 $ for all $ i \in \{ 1, \ldots, n \} $, we have $ I_1 \cap I_2 = \varnothing $ and it makes sense to define 
	$ y_i := x_i - c_i $ for $ i \in I_1 $ and $ z_i := x_i - c_i $ for $ i \in I_2 $. 
	This provides
	\[
		\ux^A - \rho \ux^B
		=  \prod_{i \in I_1} (y_i + c_i)^{A_i} - \rho \prod_{i \in I_2} (z_i + c_i)^{B_i}
		= \sum_{\alpha \in \IZ_{\geq 0}^{|I_1|}} \lambda_\alpha (\underline{c}) \underline{y}^\alpha -  \sum_{\beta \in \IZ_{\geq 0}^{|I_2|}} \mu_\beta (\underline{c}) \underline{z}^\beta ,
	\]
	where the coefficients $ \lambda_\alpha (\underline{c}), \mu_\beta (\underline{c}) \in \overline{K} $ fulfill $ \lambda_{A}(\underline{c}) = 1 $, $ \mu_B(\underline{c}) = \rho $ and $ \lambda_\alpha (\underline{c}) = \mu_\beta (\underline{c}) = 0 $ if $ |\alpha| \geq |A| $ and $ \alpha \neq A $, resp.~if $ |\beta| \geq |B| $ and $ \beta \neq B $ 
	(and we use the obvious notation $ \underline{c}, \underline{y}, \underline{z} $).
	In order to have $ \ord_{\mathfrak{n}}(g) > \min\{ |A|,|B|\} $, 
	all terms $ y^C z^D $ with $ |C|+|D| \leq  \min\{ |A|,|B|\}  $ have to cancel out.  
	This is impossible since the variables appearing in the products are disjoint and $ \lambda_{A}(\underline{c}) \mu_B(\underline{c}) \neq 0 $. 
	Thus, we arrived to a contradiction and (1) follows.
		
	Let us come to part (2). 
	Set $ d := \maxord(X) = \min \{ |A|,|B| \}$. 
	The condition $ D_I = V(x_i \mid i \in I )\subseteq\Maxord(X) $ is equivalent to 
	\[ 
		\ux^A - \rho \ux^B \in \langle x_i \mid i \in I \rangle^d \setminus  \langle x_i \mid i \in I \rangle^{d+1} .
	\]
	The latter is equivalent to $ \sum_{i\in I} A_i \geq d $ and $ \sum_{i\in I} B_i \geq d $, and equality has to hold for one of them.
	Hence, (2) follows.}
\end{proof}

\begin{example}
	Let $ g = x_1^3 x_2   - x_3^3 x_4^4 $ and set $ X := V(g) \subset \IA_K^4 $.
	Using the previous lemma, we get that $ \maxord(X) = 4 $,
	$ V(x_1, x_2, x_4) \subseteq\Maxord (X) $, 
	while $ V (x_1, x_2, x_3) \not\subseteq\Maxord(X) $.
	
	Let us blow up with center $ D := V(x_1, x_2, x_3 , x_4 ) \in \Maxord(X) $, the origin of $ \IA_K^4 $.
	In the $ X_3 $-chart, we have $ (x_1, x_2, x_3, x_4) =  (x_1'x_3', x_2'x_3', x_3', x_3'x_4') $
	and the total transform of $ g $ is   
	$ x_3'^4 (x_1'^3 x_2' - x_3'^3 x_4'^4) $.
	We obtain essentially the same binomial and no improvement is detected. 
	The reason for this is that the center has been chosen too small.  
	
	We leave it as an exercise to the reader 
	to verify that the maximal order decreases at every chart after blowing up with center $ V (x_1,x_2,x_4) $.
\end{example}

Let us now describe the method for choosing the center for a binomial 
using the locus of maximal order.

\begin{Constr}
	\label{Constr:center_max_ord}
	Let $ f = \ux^C (\ux^A - \cc \ux^B) \in K[\ux] $
	with $ \cc \in K^\times $ and $ A, B, C \in \IZ_{\ge 0}^n $ 
	such that $ A_i B_i = 0 $ for all $ i \in \{ 1, \ldots, n \} $.
	Set $ g : = \ux^A - \cc \ux^B $.
	Assume that hypothesis \eqref{eq:cond_not_loc_mon} holds.
	We choose $ I \subseteq \{ 1, \ldots, n \} $ such that 
	\[ 
 		\min \big\{\sum_{i\in I} A_i, \sum_{i\in I} B_i \big\} =
	\min \big\{ |A|, |B| \big\}  
	\]
	and we require additionally that
	\begin{equation}
	\label{eq:center_not_too_small}
		\forall \, j \in I 
		\, : \,
		\min \big\{\sum_{i\in I\setminus \{ j \}} A_i, \sum_{i\in I\setminus \{ j \}} B_i \big\} <
		\min \big\{ |A|, |B| \big\}.  
	\end{equation}

	Then, the center for the next blowup is $ D_I  := V (x_i \mid i \in I ) $.
\end{Constr}

By Lemma~\ref{Lem:Center_in_max_ord}, the center $ D_I $ is contained in the maximal order locus of $ V(g) $.
On the other hand, \eqref{eq:center_not_too_small} guarantees that $ D_I $ is not too small so that an improvement can be detected.

\begin{example}\label{ex1:maxord}\label{running_example1}
	Let $ f =  \ux^A - \ux^B = x_1^3 x_2^2 - x_3^5 x_4 \in K[x_1,x_2,x_3,x_4] $.
	Since $ \min \{ |A| , |B| \} = |A| =  5 $, we have 
	$ \{ 1, 2 \} \subseteq I $ for every $ I \subseteq\{ 1, \ldots, 4 \} $ fulfilling the conditions of Construction~\ref{Constr:center_max_ord}.
	Furthermore, $ I' := \{ 1, 2, 3, 4 \} $ does not fulfill \eqref{eq:center_not_too_small} for $ j = 4 $. 
	Therefore, the unique center determined by Construction~\ref{Constr:center_max_ord} is $ V (x_1, x_2, x_3 ) $.
\end{example}

Clearly, the subset $ I \subseteq\{ 1, \ldots, n \} $ is not unique  in general and we may have to make a choice, 
as the following example shows. 
As explained in the introduction, 
we do not require that our procedure provides a global monomialization of $ V(f) $. Therefore we may allow to make choices as long as we can prove the termination
of the resulting procedure (Proposition~\ref{Prop:max_ord_terminates}).

\begin{example}
	\label{Ex:max_ord_choice}
	Let $ f =  x_1 x_2^2 - x_3^3 x_4^2 x_5 \in K[x_1, \ldots, x_5] $.
	Since no $ x_i $ can be factored in $ f $, we have $ g = f  $.
	The maximal order of $ g $ is three and 
	$ D_1 := V (x_1, x_2, x_3) $ and $ D_2 := V (x_1, x_2, x_4, x_5) $ are 
	the possible choices for the center following Construction~\ref{Constr:center_max_ord}.

	\begin{enumerate}
		\item 
		Blow up with center $ D_1 $.
		In the $ X_3 $-chart, we have 
		\[
			(x_1, x_2, x_3, x_4, x_5) =  (x_1'x_3', x_2'x_3', x_3', x_4', x_5') .
		\] 
		Hence, the total transform of $ f $ is $ f = x_3'^3 ( x_1' x_2'^2 - x_4'^2 x_5' ) =  x_3'^3 g' $,
		where we define $ g' :=  x_1' x_2'^2 - x_4'^2 x_5' $. 
		(Note that $ g' $ fulfills the property that no $ x_i' $ divides $ g' $.)
		We have $\maxord(g) = 3 = \maxord(g')$ and $|B'| = 3 = |A| < |B| = 6$.
		
		\item 
		Blow up with center $ D_2 $. In the $ X_4 $-chart, we have 
		\[  
			(x_1, x_2, x_3, x_4, x_5) =  
			( \widetilde x_1 \widetilde x_4, \widetilde x_2 \widetilde x_4, \widetilde x_3 , \widetilde x_4, \widetilde x_4 \widetilde x_5).
		\]
		(For a better distinction to (1), we use $ \widetilde * $ instead of $ *' $ for the coordinates here.) 
		The total transform of $ f $ is $ f  = \widetilde x_4^3 ( \widetilde x_1 \widetilde x_2^2 - \widetilde x_3^3 \widetilde x_5 ) $. 
		Thus, we set $ \widetilde g := \widetilde x_1 \widetilde x_2^2 - \widetilde x_3^3 \widetilde x_5 $. 
		We get $\maxord(\widetilde g) = 3 = \maxord(g) $ and $ |\widetilde B| = 4   < 6 = |B| $.
		(The situation in the $ X_5 $-chart is analogous.)
		
	\end{enumerate} 

	On the other hand, in both cases, one can show that 
	the maximal order is strictly smaller than three if we consider the $ X_1 $- or the $ X_2 $-chart.
\end{example}

For the general case, we have to introduce a measure which detects the improvement.

\begin{defi}
	\label{Def:inv}
	Let $ g = \ux^A - \cc \ux^B \in K[\ux] $
	with $ \cc \in K^\times $ and $ A, B \in \IZ_{\ge 0}^n $ 
	such that $ A_i B_i = 0 $ for all $ i \in \{ 1, \ldots, n \} $.
	We define
	\[
		\inv(g) := \big(  \, \min\{ |A|, |B| \}, \, \max \{ |A|, |B| \} \, \big) \in \IZ_{\ge 0}^2 .
	\]
	Here, we equip $ \IZ_{\ge 0}^2 $ with the lexicographical ordering $ \ge_\lex $.
\end{defi}

In the above example, we have 
$ \inv(g) = (3, 6) $, $ \inv(g') = (3,3) <_\lex  \inv(g) $,
and $ \inv(\widetilde g) = (3,4) <_\lex \inv(g) $.

In fact,
$ \inv(g) = (\maxord(g) , \maxord(g) \cdot \delta(g)) $, where $ \delta(g) $ is a known 
secondary invariant to measure the complexity of a given singularity, e.g., 
see \cite[p.~120, where it is called $ \gamma $]{Hiro_Bowdoin}, \cite[Theorem~3.18]{InvDim2} or \cite[Corollary~5.1]{CJSc}.

\begin{Prop}
	\label{Prop:max_ord_terminates} 
	Let $ f = \ux^C (\ux^A - \cc \ux^B) \in K[\ux] = K [x_1, \ldots, x_n]$
	with $ \cc \in K^\times $ and $ A, B, C \in \IZ_{\ge 0}^n $ 
	such that $ A_i B_i = 0 $ for all $ i \in \{ 1, \ldots, n \} $.
	Set $ g : = \ux^A - \cc \ux^B $.
	Let $ \pi \colon \Bl_{D_I} (\IA_K^n) \to \IA_K^n $ be the blowup in a center 
	$ D_I $, which fulfills the properties as in Construction~\ref{Constr:center_max_ord}. 
	For every standard chart $ U_{x_j} := D_+ (X_j) \cong \IA_K^n $,
	$ j \in I $, we have 
	\[
		\inv(g') <_\lex \inv(g),
	\]
	where $ f = \ux'^{C'} (\ux'^{A'} - \cc \ux'^{B'}) \in K[\ux']  $ with $ A_i' B_i' = 0 $ for all $ i \in \{ 1, \ldots, n \} $, 
	$ g' := \ux'^{A'} - \cc \ux'^{B'} $ is the {\em strict transform} of $ g $,
	and $ (\ux') = (x_1', \ldots, x_n' ) $ are the coordinates in $ U_{x_j} $.  
	
	In particular, the local monomialization process obtained by choosing the centers as in Construction~\ref{Constr:center_max_ord} terminates.
\end{Prop}
	

\begin{proof}
	If hypothesis \eqref{eq:cond_not_loc_mon} does not hold for $ f $, 
	then $ f $ is already locally monomial and there is nothing to show.
	Hence, suppose that \eqref{eq:cond_not_loc_mon} holds. 
	Without loss of generality, we have $|A| \le |B|$. 
	After relabeling the variables $(x_1,\ldots,x_n)$ 
	we may assume that 
	\[ 
		\{ i \in \{ 1, \ldots , n\} \mid A_i \neq 0 \} 
		= \{ 1, \ldots, m \} , 
	\]
	for some $ m < n $. 
	Thus, we have $ \{ 1, \ldots, m \} \subseteq I $,
	i.e., $ D_I \subseteq  V(x_1, \ldots , x_m) $.
	
	Let us consider the $ X_i $-chart of the blowup with center $ D_I $.
	We distinguish two cases, $ i \leq m $ and $ i > m $.
	
	Assume that $ i\le m$, i.e., $A_i \neq 0$.
	Using the notation of the proposition, we have
 	\[ 
 		g' = \ux'^{A} (x_i')^{-A_i} - \cc \ux'^{B} (x_i')^{|B_I|-|A|} \in K[\ux'] ,
 	\]
 	where $ |B_I| := \sum_{j\in I } B_j $.
 	Notice that $ |B_I|-|A| \geq 0 $. 
	Our hypothesis $ |A| \leq |B| $ implies that $\inv(g') = (|A|-A_i, |B|+|B_I|-|A|)$ in this case.
	Since $A_i \neq 0$, we have $|A|>|A|-A_i$ and hence $\inv(g')<_\lex \inv(g)$.
	
	Now, suppose that $i \in \{m+1,\ldots n\}$, i.e., $B_i\neq 0$.
	We get 
	\[  
		g' =  \ux'^{A'} - \cc \ux'^{B'}  
		=  \ux'^{A} - \cc \ux'^{B} x_i'^{|B_I|-|A|-B_i} \in K [\ux'].
	\]
	First, we observe that $ \min \{ |A'|, |B'| \} \leq |A'| = |A| $. 
	If the inequality is strict, we obtain $ \inv(g') <_\lex \inv(g) $ as desired.
	Hence, let us assume that  $ \min \{ |A'|, |B'| \} = |A'| = |A| $. 
	The claim follows if we can show $|B'| < |B|$.
	By \eqref{eq:center_not_too_small} and the hypothesis $ |A|\leq |B| $, 
	we have that $|B_I|-B_i<|A|$ and therefore we get
	$|B'| = |B| +|B_I| -|A| - B_i < |B|$ and in particular $\inv(g')<_\lex \inv(g)$. 
	
	Since the improvement of $ \inv(.) $ is strict and since 
	$ \inv(.) $ takes values in $ \IZ_{\ge 0}^2 $, 
	the local monomialization procedure using centers of the kind in Construction~\ref{Constr:center_max_ord} 
	ends after finitely many steps.	
\end{proof}

Given a binomial in $ K [\ux] = K [x_1, \ldots, x_n ] $, 
we provide in Algorithm~\ref{cc4} a method to determine a subset $ I \subseteq \{ 1, \ldots, n \} $ fulfilling the conditions of
Construction~\ref{Constr:center_max_ord}.  
Therefore, $ V ( x_i\,|\, i \in I ) $ will be our center contained in the maximal order locus of the binomial.

\begin{algorithm}
	\caption{compute center (in the locus of maximal order)}
	\begin{algorithmic}[1]
		\INPUT list $ M $, $\mbox{mode}= 1 $,
		where $ M $ is of the same form as $ L[i] $ in Remark~\ref{Rk:Algo1}
		\OUTPUT $ I \subseteq \{ 1, \ldots, n \} $ such that $ V ( x_i \,|\, i \in I ) $ is the next center in the monomialization process

		\State $ (A,B,C,E) = M[1] $
		\State  $ I = \varnothing  $, $ J = \varnothing $ 
		\State $ \mathcal A = \{ i \mid  A_i > 0 \} $, $ \mathcal B = \{ i \mid B_i > 0 \} $
		\If{$|A|<|B|$}
		\State  $ I = \mathcal A $
		\For{$ b \in \mathcal B$}
		\State $ J = J \cup \{b\}$
		\If{$\sum_{j \in J } B_j \ge |A|$}
		\State \textbf{break}
		\EndIf
		\EndFor
		\For{$ i \in J$}
		\If{$\sum_{j \in J \setminus  \{i\}}B_j \ge |A|$}
		\State $ J = J \setminus \{i\}$
		\EndIf
		\EndFor
		\State $I = I \cup J $
		\State \Return $ I $
		
		\ElsIf{$|B|<|A|$}
		\State  $ I = \mathcal B $
		
		\For{$ a \in \mathcal A $}
		\State $ J = J \cup \{a\}$
		\If{$\sum_{j \in J } A_j \ge |B|$}
		\State \textbf{break}
		\EndIf
		\EndFor
		\For{$ i \in J$}
		\If{$\sum_{j \in J \setminus  \{i\}} A_j \ge |B|$}
		\State $ J = J \setminus \{i\}$
		\EndIf
		\EndFor
		\State $I = I \cup J $
		\State \Return $ I $
		
		\ElsIf{$|A|=|B|$}
		\State $ I = \mathcal A \cup \mathcal B $
		\State\Return $ I $
		\EndIf
	\end{algorithmic}
	\label{cc4}
\end{algorithm}

\begin{Rk}[Algorithm~\ref{cc4}] 
	\label{Rk:Algo_cc4}
	The {\em input}  
	is:
	\begin{itemize}
		\item 
		a list $ M $,
		which represents the data of a chart
		and hence is of the same form as $ L[i] $ in 
		Remark~\ref{Rk:Algo1}
		
		\item 
		the integer $ \mbox{mode} = 1  $,
		which tells us to choose the center as in Construction~\ref{Constr:center_max_ord};
	\end{itemize}  
	The {\em output} of Algorithm~\ref{cc4} is the index set $ I \subseteq \{ 1, \ldots, n \} $ determining the next center for the monomialization procedure.
	
	First, we initialize the data (lines 1--3).
	We introduce the exponents of the binomial $ \ux^C ( \ux^A - \rho \ux^B ) $ in the given chart. 
	Further, we introduce two auxiliary sets $ I $ and $ J $, where $ I $ will become the output set. 
	
	Then we perform a case distinction depending on whether 
	$ |A| < |B| $ (lines 4--14), or $ |A|>|B| $ (lines 15--25), or $ |A|=|B| $ (lines 26--28).
	
	Suppose $ |A|<|B| $.	
	Then, the maximal order of $ \ux^A - \rho \ux^B $ is $ |A| $
	and every $ i $ with $ A_i > 0 $ contributes to the index set of the center (line 5).
	After that, we sum up the elements of $ B $
	until the resulting sum is $ \geq |A|$ (lines 6--9).
	Within this, we collect in $ J $ the indices $ j $ with $ B_j > 0 $ appearing in the sum.
	At this moment, the index set $ \widetilde I := I \cup J $ fulfills the first condition 
	of Construction~\ref{Constr:center_max_ord}, 
	$  
	\min \big\{\sum_{i\in \widetilde I} A_i, \sum_{i\in \widetilde  I} B_i \big\} =
	\min \big\{ |A|, |B| \big\}  .
	$
	But the second condition \eqref{eq:center_not_too_small} does not necessarily hold,
	i.e., the number of elements in $ \widetilde I $ might be too large.
	Hence, we remove step-by-step elements from $ J $, without destroying first condition, 
	until \eqref{eq:center_not_too_small} holds (lines 10--12).
	
	The case $|A|>|B| $ is analogous, we only have to interchange the role of $ A $ and $ B $.
	Finally, if $ |A|=|B| $, all variables appearing in the binomial with non-zero exponent have to be contained in the ideal of the center. 
\end{Rk}

Clearly, the choice of center depends on the ordering of the variables.
In Example~\ref{Ex:max_ord_choice}, we obtain the center $ D_1 $ if we choose the ordering $ (x_1, x_2, x_3, x_4, x_5) $,
while we get $ D_2 $ for the ordering $ (x_1, x_2, x_4, x_5, x_3) $
(using the notation of the example).

\smallskip 

\section{Centers of codimension two}
\label{theory:codim2}

Let us come to the second method for choosing the center.
Recall that $ f = \ux^C (\ux^A - \cc \ux^B) $ and $ g = \ux^A - \cc \ux^B $ such that no $ x_i $ can be factored from $ g $. 
In contrast to the previous method, 
we may neglect the connection to the singularities of $ V(g) $ and choose centers of minimal codimension. 
This has the benefit that we reduce the number of charts which we have to control after a single blowup. 
Hence, the idea is to take $ i , j \in \{ 1, \ldots, n \} $ such that $ A_i \neq 0 $ and $ B_j \neq 0 $, which provides the center $ D = V (x_i, x_j ) $.
An additional requirement, which we shall need in order to detect an improvement after the blowup, is that the exponents $ A_i $ and $ B_j $ are maximal among the possible choices.

\begin{example}\label{ex1:codim2}(cf.~Example~\ref{running_example1}) \label{running_example2}
	Let $X=V(f)$ be the  hypersurface described by the binomial
	$ 
	f=x_1^3x_2^2-x_3^5x_4 \in K[x_1,x_2,x_3,x_4].
	$ 
	The maximal exponents appearing in the monomials are 
	$ A_1 = 3 $ and $ B_3 = 5 $.
	Therefore, we will choose $ D = V(x_1,x_3) $ as the center for the 
	next blowup.\\
	In comparison to Example~\ref{ex1:maxord} we see that we can reduce the 	number of successor charts by choosing a center of codimension 2.
\end{example}

Note that  the maximal order is not an appropriate measure to detect the improvement along a blowup of the given type. 

\begin{example}
	\label{Ex:codim_2_maxord_increases}
	Let $ f = g = x_1^a x_2^a x_3^a - x_4^b \in K[x_1, x_2, x_3, x_4] $ with $ a, b \in \IZ_{\ge 2} $ such that $ a < b < 2a $.
	(For example, take $ a = 3 $ and $ b = 5 $).
	Observe that $ \maxord(g) = b $. 
	We choose the center $ V(x_1, x_4) $.  
	In the $ X_1 $-chart, we have 
	\[ 
		(x_1 ,x_2, x_3, x_4 ) = (x_1', x_2, x_3, x_1'x_4')
	\] 
	and the total transform of $ f $ is $ f = x_1'^a (x_2^a x_3^a - x_1'^{b-a}x_4'^b) = x_1'^a g' $, where we define 
	\[ 
		g' := \ux'^{A'} - \ux'^{B'} := x_2^a x_3^a - x_1'^{b-a}x_4'^b .
	\]  
	Since $ a < b $, we have $|B'| = 2b-a > b $,
	while $ b < 2a $ implies $|A'|= 2a > b$.   
	Therefore, we have $ \maxord(g') > \maxord(g) $,
	which is not a surprise since the center $ V(x_1, x_4) $ is not contained in $ \Maxord(g) $. 
	
	Note that the same happens if we blow up one of the other 
	reasonable centers of codimension two, which are $ V(x_2,x_4) $ and $ V(x_3,x_4) $. 
\end{example}

We use $ \iota(g) = \left( \alpha(g), \mathfrak{a}(g), \beta(g), \mathfrak{b}(g)\right) \in \mathbb{Z}_{\ge 0}^4 $
	of Definition~\ref{Def:Invariant_codim_2} to deduce the termination for the present method of monomialization.
In Example~\ref{Ex:codim_2_maxord_increases}, we have 
$ \iota(g') = (a,2, b, 1) <_\lex (a,3,b,1) = \iota(g) $. 
In general, not any codimension two center provides an improvement of $ \iota(.) $.

\begin{example}
	Consider the binomial $ f = g = x_1^a - x_2^b x_3^c \in K[x_1, x_2, x_3] $ 
	with $ a, b, c \in \IZ_{\ge 2} $ and $ b \leq c $. 
	We have
	\[
		\iota(g) = \left\{
		\begin{array}{cl}
		(a,1,c,1), & 
		\mbox{if } b < c, 
		\\[3pt]
		(a,1,c,2), & 
		\mbox{if } b = c.
		\end{array}  
		\right.
	\] 
	Suppose that $ b < c $. 
	Let us blow up with center $ V(x_1, x_2) $,
	which does not fulfill the additional hypothesis that the corresponding exponents in $ g $ are maximal.
	For simplicity, we define $ m := \min \{ a, b \} $. 
	In the $X_2$-chart, the total transform of $ f $ is $ f= x_2'^{m} (x_1'^a x_2'^{a-m} - x_2'^{b-m}x_3^c )$, which provides
	$ g' = x_1'^a x_2'^{a-m} - x_2'^{b-m}x_3^c $ and 
	$ \iota(g') = (a,1,c,1) = \iota(g) $. 
\end{example}

In order to guarantee a decrease of $ \iota(.) $, 
we choose the center as follows:

\begin{Constr}
	\label{Constr:center_codim2}
	Let $ f = \ux^C (\ux^A - \cc \ux^B) \in K[\ux] $
	with $ \cc \in K^\times $ and $ A, B, C \in \IZ_{\ge 0}^n $ 
	such that $ A_i B_i = 0 $ for all $ i \in \{ 1, \ldots, n \} $.
	Set $ g : = \ux^A - \cc \ux^B $.
	Assume that hypothesis \eqref{eq:cond_not_loc_mon} holds.
	If $ \iota(g) = (\alpha,\mathfrak a, \beta, \mathfrak b) \in \IZ_{\ge 0}^4 $,
	then we choose $ j_1, j_2 \in \{ 1, \ldots, n \} $ such that  
	$ A_{j_1} = \alpha $ and $  B_{j_2} = \beta $.
	The center for the next blowup is then $ D_{I} = V(x_{j_1},x_{j_2}) $, for $ I = \{ j_1, j_2 \} $.
\end{Constr}

Already in Example~\ref{Ex:codim_2_maxord_increases},
we have seen that the center described in the above construction is not unique. 
But, analogous to Proposition~\ref{Prop:max_ord_terminates}, we can prove the following result.

\begin{Prop}
	\label{Prop:codim2_terminates} 
	Let $ f = \ux^C (\ux^A - \cc \ux^B) \in K[\ux] = K [x_1, \ldots, x_n]$
	with $ \cc \in K^\times $ and $ A, B, C \in \IZ_{\ge 0}^n $ 
	such that $ A_i B_i = 0 $ for all $ i \in \{ 1, \ldots, n \} $.
	Set $ g : = \ux^A - \cc \ux^B $.
	Let $ \pi \colon \Bl_{D_I} (\IA_K^n) \to \IA_K^n $ be the blowup in a center 
	$ D_I $, which fulfills the properties as in Construction~\ref{Constr:center_codim2}. 
	For every standard chart $ U_{x_j} := D_+ (X_j) \cong \IA_K^n $,
	$ j \in I $, we have 
	\[
	\iota(g') <_\lex \iota(g),
	\]
	where $ f = \ux'^{C'} (\ux'^{A'} - \cc \ux'^{B'}) \in K[\ux']  $ with $ A_i' B_i' = 0 $ for all $ i \in \{ 1, \ldots, n \} $, 
	$ g' := \ux'^{A'} - \cc \ux'^{B'} $,
	and $ (\ux') = (x_1', \ldots, x_n' ) $ are the coordinates in $ U_{x_j} $.  
	
	In particular, the local monomialization process obtained by choosing the centers as in Construction~\ref{Constr:center_codim2} terminates.
\end{Prop}

\begin{proof}
	We may assume $ j_1 = 1 $ and $ j_2 = 2 $
	after relabeling the variables $ (x_1, \ldots, x_n) $. 
	Hence, the center is $ V(x_1, x_2) $. 
 	Since the exponents $ A_+ := (A_3, \ldots, A_n) $ and $ B_+ := (B_3, \ldots, B_n) $ are not changed by the blowup, 
 	we use the abbreviation
 	\[
 		g = x_1^{\alpha} \ux_+^{A_+} - \cc x_2^{\beta} \ux_+^{B_+},
 	\]
 	where $ \iota(g) = (\alpha,\mathfrak a, \beta, \mathfrak b)  $ and $ \ux_+ := (x_3, \ldots, x_n) $.
 	Without loss of generality, we assume $\alpha \le \beta$. 
 	
 	In the $X_1$-chart, the total transform of $ f $ provides 
 	$g'= \ux'^{A'} - \cc \ux'^{B'}  := 
 	\ux_+'^{A_+} - \cc x_1'^{\beta-\alpha}x_2'^{\beta} \ux_+'^{B_+}$.
 	Therefore, $\alpha(g') < \alpha = \alpha(g) $, if $\mathfrak a = \mathfrak a (g) = 1$, or 
 	$ ( \alpha(g'), \mathfrak a(g')) = ( \alpha, \mathfrak a -1)$ otherwise. 
 	So, we get $\iota(g') <_\lex \iota(g)$. 
 	
 	Let us consider the $X_2$-chart of the blowup. 
 	In there, we obtain 
 	$g'= \ux'^{A'} - \cc \ux'^{B'}  :=
 	x_1'^{\alpha} \ux_+'^{A_+} - \cc x_2'^{\beta - \alpha} \ux_+'^{B_+}$.
	In order to show that $\iota(.)$ improves, we first notice that $ (\alpha(g'), \mathfrak a (g') )  = (\alpha(g), \mathfrak a (g) ) $ since $ A' = A $.
	Analogous to the $X_1$-chart, we have $\beta(g') < \beta =\beta(g) $, if $\mathfrak b = \mathfrak b (g) = 1$, or $ (\beta(g'), \mathfrak b(g')) = (\beta, \mathfrak b -1 )$ otherwise.
	Hence, we have in all cases that $\iota(g') <_\lex \iota (g)$.

 	Since the improvement of $ \iota(.) $ is strict and since 
 	$ \iota(.) $ takes values in $ \IZ_{\ge 0}^4 $, 
 	the local monomialization procedure using centers of the kind in Construction~\ref{Constr:center_codim2} 
 	ends after finitely many steps.
\end{proof}

The blowup $ \pi \colon \Bl_{D_I} (\IA_K^n) \to \IA_K^n $ with center $ D_I $ chosen following Construction~\ref{Constr:center_codim2}
is not necessarily an isomorphism outside the singular locus of $ V (g) $.

\begin{example} 
	Let $f = g = x_1 x_2 -x_3^5 \in K[x_1,x_2,x_3] $.
	Construction~\ref{Constr:center_codim2} provides the possible centers $ V(x_1, x_3) $ and $ V(x_2,x_3) $.
	Both are not contained in the singular locus of $ V(g) $, which is 
	$ \Sing (V(g)) = V (x_1,x_2,x_3) $.
	Therefore, the potential centers are strictly larger than the singular locus and the corresponding blowup morphisms are not an isomorphisms outside of the singular locus. 
\end{example}

In Algorithm~\ref{cc3}, we provide an implementation of Construction~\ref{Constr:center_codim2}
to choose a center of codimension $ 2 $.

\begin{algorithm}
	\caption{compute center (with codimension 2)}
	\begin{algorithmic}[1]
		\INPUT list $ M $, $\mbox{mode}= 2 $,
		where $ M $ is of the same form as $ L[i] $ in Remark~\ref{Rk:Algo1}
		\OUTPUT $ I \subseteq \{ 1, \ldots, n \} $ such that $ V ( x_i \,|\, i \in I ) $ is the next center in the monomialization process
		\State $ (A,B,C,E) = M[1] $
		\State $ \alpha = \max\{ A_i \} $, $ \beta = \max\{ B_i \} $
		\State  $ I = \{ \, \min\{ i \mid A_i = \alpha \}, \, \min\{ i \mid B_i = \beta \} \, \} $
		\State\Return  $ I $
	\end{algorithmic}
	\label{cc3}
\end{algorithm} 

\begin{Rk}[Algorithm~\ref{cc3}]
	The input and the output are of the same form as in Algorithm~\ref{cc4} (see Remark~\ref{Rk:Algo_cc4}) with the only difference that the input {mode} is $ 2 $.
	We initialize the data,
	by fixing the names of the exponents of the binomial $f=\ux^C(\ux^A-\cc\ux^B)$ of this chart
	and determining the maximal exponents $ \alpha $ resp.~$ \beta $ appearing on each side. 
	Then,
	we choose $ i$ and $ j $ minimal in $  \{ 1, \ldots, n \} $ such that $ A_i =\alpha $ and $ B_j =\beta $ achieve the maximal values. 
\end{Rk}

Note that we cannot have $ \alpha = 0 $ or $ \beta = 0$ in this algorithm since we tested whether \eqref{eq:ex_loc_mon_cond} holds before applying  \textit{compute\_center} in
Algorithm~\ref{main} (line 3) resp.~in Algorithm~\ref{transformation} (line 11).

\smallskip

\section{Centers of minimal codimension contained in the singular locus}
\label{theory:singloc}

The third variant is a mixture of the first two.
Namely, we want to choose centers as large as possible (as in Construction~\ref{Constr:center_codim2}), 
but we require additionally that along the monomialization process 
the centers are contained in the singular locus of the factor, which we obtain after factoring the monomial part, (similar to Construction~\ref{Constr:center_max_ord}).
If this singular locus is empty, we follow the method of section~\ref{theory:codim2} {(Construction~\ref{Constr:center_codim2})}.
Within this, we have to distinguish several cases.

%
%

\begin{Constr}
	\label{Constr:center_non-monomial}
	Let $ f = \ux^C (\ux^A - \cc \ux^B) \in K[\ux] $
	with $ \cc \in K^\times $ and $ A, B, C \in \IZ_{\ge 0}^n $ 
	such that $ A_i B_i = 0 $ for all $ i \in \{ 1, \ldots, n \} $.
	Set $ g : = \ux^A - \cc \ux^B $.
	Assume that hypothesis \eqref{eq:cond_not_loc_mon} holds.
	Let $ \iota(g) = (\alpha,\mathfrak a, \beta, \mathfrak b) \in \IZ_{\ge 0}^4 $.
	\begin{itemize}
		\item[(i)]
		If $ \min \{ \alpha, \beta \} \geq 2 $ or $ \min \{ |A|, |B| \} = 1 $,  
		choose $ D_I  = V(x_{j_1},x_{j_2}) $, for $ I = \{ j_1, j_2 \} $ as in Construction~\ref{Constr:center_codim2}, for the center of the blowup.
	
		\smallskip 
		
		\item[(ii)] 
		If $ \alpha = 1 $, $ \beta \ge 2 $ and $ \min\{ |A|,|B|\} \ge 2 $,
		 choose $ j_1,  j_2 , j_3 \in \{ 1, \ldots, n \} $ with $ A_{j_1} = A_{j_2} = 1 $ and $ B_{j_3} = \beta $.
		The center of the next blowup is $ D_I = V (x_{j_1}, x_{j_2}, x_{j_3}) $,
		for $ I = \{ j_1, j_2, j_3 \} $.	
		
		\smallskip 
		
		\item[(iii)] 
		If $ \alpha \ge 2 $, $ \beta = 1 $ and $ \min\{ |A|,|B|\} \geq 2 $,
		 choose $ j_1,  j_2 , j_3 \in \{ 1, \ldots, n \} $ with $ A_{j_1} = \alpha $ and $ B_{j_2} = B_{j_3} = 1 $.
		The center of the next blowup is $ D_I = V (x_{j_1}, x_{j_2}, x_{j_3}) $,
		for $ I = \{ j_1, j_2, j_3 \} $.	
		
		\smallskip
		
		\item[(iv)] 
		If $ \alpha = \beta = 1 $ and $ \min\{ |A|,|B|\} \geq 2 $,
		 choose $ j_1,  j_2 , j_3, j_4 \in \{ 1, \ldots, n \} $ with $ A_{j_1} =  A_{j_2} = B_{j_3} = B_{j_4} = 1 $.
		The center of the next blowup is $ D_I = V (x_{j_1}, x_{j_2}, x_{j_3}, x_{j_4}) $,
		for $ I = \{ j_1, j_2, j_3, j_4 \} $.	
		
	\end{itemize}
	We say {\em $ (f,g) $ is in case ($*$)} if condition ($*$) is fulfilled,
	where $ * \in \{ \mbox{i}, \mbox{ii}, \mbox{iii}, \mbox{iv} \} $. 
\end{Constr}

Observe that for $ f = x_1 (x_1 - x_2 x_3) $ we are in case (i),
while $ f = x_1 (x_1 - x_2) $ is monomial. 
Furthermore, $ f = x_1 x_2 x_3 - x_4 x_5 x_6 $ is case (iv) and there are several choices for the center.

\begin{Prop}
	\label{Prop:singlocus_terminates} 
	Let $ f = \ux^C (\ux^A - \cc \ux^B) \in K [x_1, \ldots, x_n]$
	with $ \cc \in K^\times $ and $ A, B, C \in \IZ_{\ge 0}^n $ 
	such that $ A_i B_i = 0 $ for all $ i \in \{ 1, \ldots, n \} $.
	Let $ g  = \ux^A - \cc \ux^B $ and $ \iota(g) = (\alpha,\mathfrak a, \beta, \mathfrak b) \in \IZ_{\ge 0}^4 $.
	Let $ \pi \colon \Bl_{D_I} (\IA_K^n) \to \IA_K^n $ be the blowup in a center 
	$ D_I $, which fulfills the properties as in Construction~\ref{Constr:center_non-monomial}. 
	For every standard chart $ U_{x_j} := D_+ (X_j) \cong \IA_K^n $,
	$ j \in I $, we have 
	\[
	\iota(g') <_\lex \iota(g),
	\]
	where $ f = \ux'^{C'} (\ux'^{A'} - \cc \ux'^{B'}) \in K[\ux']  $ with $ A_i' B_i' = 0 $ for all $ i \in \{ 1, \ldots, n \} $, 
	$ g' := \ux'^{A'} - \cc \ux'^{B'} $,
	and $ (\ux') = (x_1', \ldots, x_n' ) $ are the coordinates in $ U_{x_j} $.  
	
	In particular, the local monomialization process obtained by choosing the centers as in Construction~\ref{Constr:center_non-monomial} terminates.
\end{Prop}

\begin{proof}
	We show the result by going through all cases of Construction~\ref{Constr:center_non-monomial}.
	First,
	if {\em $(f,g)$ is in case (i)}, then we have $ \iota(g') <_\lex \iota(g) $ by Proposition~\ref{Prop:codim2_terminates}.

	\smallskip 
	
	Next, we assume that {\em $(f,g)$ is in case (ii)}, 
	i.e., $\alpha = 1, \beta \ge 2$ and $\min\{|A|,|B|\} \ge 2$.
	We relabel the variables, so that we have 
	$ B_{m+1} = \beta $ and 
	$\{i \in \{1,\ldots,n\} \mid A_i = 1\} = \{1,\ldots,m\}$ for some $ 2 \le m < n$.
	In particular, 
	\[ 
		g = x_1 x_2\cdots x_m - \cc \ux^B.
	\]
	Hence, without loss of generality, the center is $D_I := V(x_1,x_2,x_{m+1})$. 
	
	The $ X_1$- and the $X_2$-chart are analogous, so we consider only one of them. 
	In the $X_1$-chart, we get
	(using the notation of the statement of the proposition)
	\[
		 g'= x_2' \cdots x_m' - \cc\, x_1'^{\beta -2}\, \ux'^B  .
	\]
	We see that $ (  \alpha(g'), \mathfrak a(g') ) = (1, m -1 ) <_\lex (1,m) = ( \alpha (g) , \mathfrak a(g) ) $.
	This implies the desired decrease $\iota(g') <_\lex \iota(g) $.
	Observe that $ (  \beta(g'), \mathfrak b(g') ) = ( \beta (g) , \mathfrak b (g) ) $ did not change.

	On the other hand, using the notation $\ux^B = x_{m+1}^\beta\ux_+^{B_+}$,
	we
	obtain
	in the $X_{m+1}$-chart
	\[ 
		g' = x_1' \cdots x_m' - \cc \, x_{m+1}'^{\beta -2} \, \ux_+'^{B_+} 
		.
	\] 
	We have $ (  \alpha(g'), \mathfrak a(g') ) = ( \alpha (g) , \mathfrak a(g) ) $
	and $ \beta(g') \leq \beta(g) $. 
	Either the inequality is strict
	or we have equality and $ \mathfrak b (g') = \mathfrak b (g) - 1 $ since the power of $ x_{m+1}' $ decreased strictly. 
	In both cases, we get $\iota(g') <_\lex \iota(g) $.

	\smallskip 
	
	The case that {\em $ (f,g) $ is in case (iii)} is analogous to the previous one.
	One only has to interchange the role of $ A $ and $ B $ and take into account that $ (\alpha(g'), \mathfrak a (g')) = (\alpha(g), \mathfrak a (g) ) $ in the $ X_1 $-(resp.~$ X_2 $-)chart. 
	
	\smallskip 
	
	Finally, suppose that {\em $(f,g)$ is in case (iv)}.
	After relabeling the 
	variables we get 
	\[ 
		g= x_1\cdots x_m - \cc x_{m+1} \cdots x_{m+\ell},
	\] 
	for some 
	$2\le m<n$ and $ 2 \leq  \ell < n $ with $m+\ell \le n$.
	Without loss of generality,
	the center is $D_I = V(x_1,x_2,x_{m+1},x_{m+2})$. 
	So, we have to consider four charts. 
	In the $X_1$-chart we get
	\[ 
		g' = \ux'^{A'} - \cc \ux'^{B'} =  x_2' \cdots x_m' - \cc x_{m+1}' \cdots x_{m+\ell}' .
	\]
	This implies that 
	$ (\alpha(g'), \mathfrak a (g') ) = (\alpha(g), \mathfrak a (g)- 1 ) $
	and (since $ B' = B $) we also have 
	$ (  \beta(g'), \mathfrak b(g') ) = ( \beta (g) , \mathfrak b (g) ) $.
	In particular, we get $ \iota(g') <_\lex \iota(g) $. 
	The other three charts are analogous.

	\smallskip 
	
	Since the improvement of $ \iota(.) $ is strict in every case 
	and since 
	$ \iota(.) $ takes values in $ \IZ_{\ge 0}^4 $, 
	the local monomialization procedure using centers of the kind in Construction~\ref{Constr:center_non-monomial} 
	ends after finitely many steps.
\end{proof}

In Algorithm~\ref{cc2} we discuss an implementation for the choice of the center following Construction~\ref{Constr:center_non-monomial}.

\begin{algorithm}
	\caption{compute center of minimal codimension contained in the singular locus}
	\begin{algorithmic}[1]
		\INPUT list $ M $, $\mbox{mode}= 3 $,
		where $ M $ is of the same form as $ L[i] $ in Remark~\ref{Rk:Algo1}
		\OUTPUT $ I \subseteq \{ 1, \ldots, n \} $ such that $ V ( x_i \,|\, i \in I ) $ is the next center in the monomialization process
		\State $ (A,B,C,E) = M[1] $
		\State $ \alpha = \max\{ A_i \} $, $ \beta = \max\{ B_i \} $
		\State  $ i_1 = \min\{ i \mid A_i = \alpha \} $, 
		$ i_2 = \min\{ i \mid B_i = \beta \} $
		\State  $ I = \{ i_1, i_2 \} $
		
		\If{$\min\{ \alpha,\beta\} \ge 2$ or $\min\{|A|,|B|\} == 1 $ 
		}\hskip5em// case(i)
		\State\Return $ I $

		\ElsIf{$\min\{|A|,|B|\} \ge 2$}
		\If{$\alpha =1$ and $\beta \ge 2$}
		\Comment {case(ii)}
		\State $ I = I \cup \min\{ i \mid A_i = 1 \mbox{ and } i > i_1  \} $
		
		\ElsIf{$\alpha \ge 2 $ and $\beta =1$}
		\Comment {case(iii)}
		\State $ I = I \cup \min\{ i \mid B_i = 1 \mbox{ and } i > i_2  \} $
		
		\ElsIf{$\alpha =1$ and $\beta =1$}
		\Comment {case(iv)}
		\State $ I = I \cup \{ \,
		\min\{ i \mid A_i = 1 \mbox{ and } i > i_1  \}, \, 
		\min\{ i \mid B_i = 1 \mbox{ and } i > i_2  \} \, \}  $
		
		\EndIf
		\State\Return $ I $ 
		
		\EndIf
	\end{algorithmic}
	\label{cc2}
\end{algorithm}

\begin{Rk}[Algorithm~\ref{cc2}]
	The input and the output are of the same form as in Algorithm~\ref{cc4} (see Remark~\ref{Rk:Algo_cc4}) with the only difference that the input mode is $ 3 $.
	First, we initialize the data and determine the minimal indices $i_1 $ and $ i_2 $, for which the maximal entry of $ A $ resp.~$ B $
	is achieved, 
	where $ f = \ux^C ( \ux^A - \rho \ux^B) $ is the monomial in the given chart. 
	
	If $ V (x_{i_1}, x_{i_2} ) $ is contained in the singular locus of $ V(\ux^A - \rho \ux^B) $ 
	or if $ f $ is of the form $ f = \ux^C ( x_j - \ux^D ) $, for $ j \in \{ i_1, i_2 \} $ and $ D \in \{ A, B \} $ the corresponding element,
	then the algorithm returns $ \{ i_1, i_2 \} $ as the index set for the upcoming center (line 5--6).
	This is case (i) of Construction~\ref{Constr:center_non-monomial}. 
	
	Otherwise, we have $ \min \{ \alpha, \beta \} = 1 $ and $ \min \{ |A|, |B|\} \geq 2 $. 
	Thus we make a case distinction depending on the value of $ \alpha $ and $ \beta $ (starting line 7), 
	where we have to add an index $ i > i_1 $ for which $ A_i = 1 $ if $ \alpha = 1 $
	and analogous if $ \beta = 1 $.
	The latter guarantees the center is contained in the singular locus of $\ux^A - \rho \ux^B $. 
	This covers the missing cases (ii)--(iv) of Construction~\ref{Constr:center_non-monomial}. 	
\end{Rk}

As in Algorithm~\ref{cc3}, the case $ \min \{ \alpha, \beta \}  = 0 $ cannot appear since we tested whether the data of the chart is of the form \eqref{eq:ex_loc_mon_cond}
(line 3 Algorithm~\ref{main} and line 11 Algorithm~\ref{transformation}). 

\medskip 

If we consider the binomial of Example~\ref{running_example2}
and apply Algorithm~\ref{cc2}, we obtain the same center as using Algorithm~\ref{cc3}.
But it may happen that the two algorithms lead to different centers. 

\begin{example}
	Let $ f = x_1 x_2 + x_3^2 \in K[x_1, x_2, x_3] $. 
	Algorithm~\ref{cc2} provides the center $ V(x_1, x_2, x_3 ) $,
	while Algorithm~\ref{cc3} gives the center $ V (x_1, x_3 ) $. 
\end{example}

\smallskip 

\section{Centers of minimal codimension contained in an exceptional divisor or contained in the singular locus}
\label{theory:preimagesingloc}

Now, we present the fourth variant, where we slightly 
relax the restriction on the centers that we imposed in the previous section.
We obtain this by taking the preceding resolution process into account. 

Let $ \pi \colon \Bl_D(\IA_K^n) \to \IA_K^n $ be the blowup with center 
$ D = V(x_i \mid i \in I) $ for $ I \subseteq \{ 1 , \ldots, n \} $.
In the $ X_i $-chart $ U_i = D_+ (X_i) \cong \IA_K^n $ ($ i \in I $), 
we have the coordinates $ (\ux') = (\ux_1', \ldots, \ux_n') $
and $ \langle x_j \mid j \in I \rangle \cdot K[\ux'] = \langle x_i' \rangle \subset  K[\ux'] $.
Hence, the preimage of the center along the blowup $ \pi $ coincides with the divisor $ E := \operatorname{div}(x_i) $ in $ U_i $. 
Since a blowup is an isomorphism outside of its center, 
we have the freedom to choose any center $ D'  \subset U_i $ contained in $ E $ without losing the condition that the composition of $ \pi $ and the blowup in $ D' $ is an isomorphism outside of $ D $.
In particular, we may choose centers of codimension two as in Construction~\ref{Constr:center_codim2}.

Clearly, the previous observation extends to any finite sequence of (local) blowups of the above type. 
This motivates the following method for choosing the center.

\begin{Constr}
	\label{Constr:center_exceptional_non-monomial}
	Let $ f = \ux^C (\ux^A - \cc \ux^B) \in K[\ux] $
	with $ \cc \in K^\times $ and $ A, B, C \in \IZ_{\ge 0}^n $ 
	such that $ A_i B_i = 0 $ for all $ i \in \{ 1, \ldots, n \} $.
	Let $ (y_1, \ldots, y_m) $, $ m \leq n $, be a subsystem of distinguished variables of 
	$ (x_1, \ldots, x_n) $ such that the exceptional divisor of the 
	local monomialization procedure is given by $ \operatorname{div}(y_1 \cdots y_m) $. 
	If there is a center $ D_I = V (x_{j_1}, x_{j_2} ) $ as in Construction~\ref{Constr:center_codim2}, 
	which is also contained in the exceptional locus, 
	$ D_I \subset \operatorname{div}(y_1 \cdots y_m) $,
	then choose $ D_I $ as the center for the next blowup.
	Otherwise, we follow Construction~\ref{Constr:center_non-monomial}.
\end{Constr} 

Note that $ D_I \subset \operatorname{div}(y_1 \cdots y_m) $ is equivalent to the condition that $ x_{j_1} = y_{k_1} $ or $ x_{j_2} = y_{k_2} $, for some $ k_1, k_2 \in \{ 1, \ldots, m \} $. 

\smallskip  

As a consequence of Propositions~\ref{Prop:codim2_terminates} and \ref{Prop:singlocus_terminates}, we get the termination of the local monomialization procedure using centers given by Construction~\ref{Constr:center_exceptional_non-monomial}
(since $ \iota(.) $ decreases strictly after the blowup).

\begin{Cor}
	\label{Cor:except_singlocus_terminates}
	The local monomialization process obtained by choosing the centers as in Construction~\ref{Constr:center_exceptional_non-monomial} terminates
	for every binomial $ f \in K [\ux] $.
\end{Cor}

In Algorithm \ref{cc1} we 
present an implementation of Construction~\ref{Constr:center_exceptional_non-monomial}.
This is analogous to Algorithm~\ref{cc2} with the only difference in line 5, where we have the additional condition $ E_{i_1} + E_{i_2} > 0 $. 
The latter holds whenever $ x_{i_1} $ or $ x_{i_2} $ correspond to an exceptional divisor of a previous blowup.

\begin{algorithm}
	\caption{compute center with minimal codimension contained in an exceptional divisor or contained in the singular locus}
	\begin{algorithmic}[1]
		\INPUT list $ M $, $\mbox{mode}= 4 $,
		where $ M $ is of the same form as $ L[i] $ in Remark~\ref{Rk:Algo1}
		\OUTPUT $ I \subseteq \{ 1, \ldots, n \} $ such that $ V ( x_i \,|\, i \in I ) $ is the next center in the monomialization process
		\State $ (A,B,C,E) = M[1] $
		\State $ \alpha = \max\{ A_i \} $, $ \beta = \max\{ B_i \} $
		\State  $ i_1 = \min\{ i \mid A_i = \alpha \} $, 
		$ i_2 = \min\{ i \mid B_i = \beta \} $
		\State  $ I = \{ i_1, i_2 \} $
		
		\If{$\min\{ \alpha,\beta\} \ge 2$ or $\min\{|A|,|B|\}==1$ or $ E_{i_1} + E_{i_2} > 0 $}
		\State\Return $ I $ \Comment {exceptional or case(i)}
		
		\Else 
		\State $ I $ = compute\_center$ (M, 3) $
		\State\Return $ I $ \Comment {case(ii),(iii),(iv)}
		\EndIf 
	\end{algorithmic}
	\label{cc1}
\end{algorithm}

\smallskip

\section{A glimpse into the case of more than one binomial and non-invertible coefficients}

Let us briefly outline a method to extend our procedures to finitely many binomials
	and to a single binomial with coefficients in $ \IZ_p $.

\smallskip 

\begin{Constr}
	Let $ f_1, \ldots, f_m \in K [\ux] = K[x_1, \ldots, n] $ be finitely many binomials, where $  K $ is a field.
	In order to monomialize them, we may successively apply one of the discussed procedures to $ f_1 $, then the total transform of $ f_2 $ and so on. 
	Since each of methods terminates, we reach the case that each $ f_i $ is locally monomial. 
\end{Constr}

At the end of the last construction, the total transform of the binomials $ f_1, \ldots, f_m $ are not necessarily simultaneously locally monomial,
or in other words, the product $ f_1 \cdots f_m $ is not necessarily locally monomial, as the following example illustrates. 

\begin{example}
	\label{Ex:there_is_more_for_m>1}
	Let $ K $ be a field and consider 
	\[
		f_1 := x_1 - 1, \ \ \ f_2 := x_2 - 1, \ \ \ f_3 := x_1 x_2 - 1. 
	\]
	We claim that 
	the product $ f_1 f_2 f_3 $ is not locally monomial.
	To see this, we introduce $ y_1 := x_1 - 1 $ and $ y_2 := x_2 - 1 $.
	Then $ f_1 f_2 f_3 = y_1 y_2 ( y_1 y_2 + y_1 + y_2) $.  
\end{example}

Eventually, the task to make $ f_1, \ldots, f_m $ simultaneously locally monomial can be reduced to the problem of (locally) monomializing an element of the form 
\[
	\prod_{i=1}^m (\ux^{A(i)} - \lambda_i), 
	\ \ \ \ \lambda_i \in K^\times \mbox{ and } A(i) \in \IZ_{\ge 0}^n \mbox{ for } 1 \leq i \leq m.
\]
This problem is connected to the desingularization of arrangements of smooth subvarieties, 
which is treated over algebraically closed fields in \cite{Hu} or \cite{LiLi}, for example.
{Nonetheless, for finitely many prime characteristics $ p = \operatorname{char}(K) > 0 $ (depending on the exponents $ A(i) $) 
	the situation becomes more involved and further investigations are required.}
Since the present article focuses on the case of a single binomial, 
we do not go into the details here.

\smallskip

When successively applying a monomialization method to $ (f_1, f_2, \ldots, f_m) $,
the order in which we handle the elements has an impact on the final numbers of charts.

\begin{example}
	\label{Ex:dependenc_ordering}
	Let $ K $ be any field. 
	Consider the binomials 
	\[ 
		f_1 = v^2-y^4z , \ \ f_2=x^2y-z^3  \ \ \in K[x,y,z,v].
	\]
	
	If we use our local monomialization method with codimension two 
	centers for $ ( f_1, f_2) $
	(i.e., first for $ f_1 $ and then for the total transform of $ f_2 $),
	then the procedure needs computation in $43$ charts and $19$ of these charts 
	are final charts.
	
	On the other hand, if we take the order $ (f_2, f_1) $, then the same procedure needs only $31$ charts
	and $12$ of them are final charts. 
\end{example}

\smallskip 

We now turn out attention to the situation over $ \IZ_p $ instead of over a field.

\begin{Rk}
	\label{Rk:ZZ_p}
	We observe that our proofs for the termination of the local monomialization procedures
	rely on a study of the exponents. 
	This suggests that 
	the methods may also be used for a first steps towards a monomialization
	if we are not necessarily restricted to the situation over field $ K $. 
	Consider a binomial in $ \IZ_p [\ux] = \IZ_p [x_1, \ldots, x_n] $, say
	\[
		f = p^e \hspace{1pt} \ux^C ( \, \ux^A - \lambda \hspace{1pt} p^d \hspace{1pt} \ux^B \, ),
		\hspace{30pt} \lambda \in \IZ_p^\times, 
	\]  
	for $ d, e \in \IZ_{\geq 0} $ and $ A,B, C \in \IZ_{\ge 0}^n $ with $ A_i B_i = 0 $ for all $ i \in \{ 1, \ldots, n \} $. 
	We fix one of the procedures, which we discussed, and apply it to $ f $, considered as a binomial with coefficients in the field $ \IQ_p = \operatorname{Quot}(\IZ_p) $. 
	As we have seen, this terminates after finitely many blowups.
	Since the coefficients are in $ \IZ_p $, the resulting total transform of $ f $ is not necessarily locally monomial. 
	For example, it may appear that the total transform of $ f $ is of the form
	\[ 		
		p^e \hspace{1pt} \ux'^{C'} ( \, \ux'^{A'} - \lambda \hspace{1pt} p^d \hspace{1pt} \, ),
	\]
	for some $ A' \in \IZ_{\ge 0}^n $ with $ |A'| \geq 2 $. 
	Here, $ d, e \in \IZ_{\ge 0} $ are the same integers as at the beginning since we did not touch coefficients.
	If $ d \neq 0 $, then the monomialization process is not finished, yet.  
	
	There are at least two directions that one could follow:
	\begin{enumerate}
		\item 
		We blow up centers of the form $ \langle x_i', p \rangle $,
		where $ i $ is chosen appropriately. 
		This has the drawback that the ambient ring after the blowup is not necessarily isomorphic to a polynomial ring over $ \IZ_p $.
		More precisely, in the $ X_i' $-chart, we get
		$ \IZ_p[\ux',v]/ \langle p - x_i' v \rangle $.
		
		\smallskip 
		
		\item  
		An alternative method 
		is to make a case distinction depending on the residues of $ x_i' $ modulo $ p $. 
		If $ x_i' \equiv 0 \mod p $, we can write it as $ x_i' = p \, y_i $ for some new variables taking values in $ \IZ_p $. 
		By choosing $ i $ appropriately, we can make $ d $ decrease.
		On the other hand, if $ x_i' \not\equiv 0 \mod p $ for all $ i $, then $ \ux'^{A'} - \lambda \hspace{1pt} p^d $ is a unit.
		Thus $ f $ is monomial. 
	\end{enumerate}
\end{Rk}

\smallskip

\section{Comparing the variants}

Finally, let us compare the discussed algorithms for monomializing a binomial.
First, we analyze the complexity of the algorithms by estimating the maximal possible number of blowups needed to monomialize a given binomial. 
After that we turn out attention to explicit examples, where we compare the numbers of charts appearing along the monomialization and as well as the number of final charts. 
At the end, we briefly look at the question, whether the different choice for the centers in a fixed method have an impact on the resulting numbers.

\begin{Rk}
	We can interpret the process of blowing up as a tree structure. 
	The vertices correspond to the charts, 
	where we put the original data on level 0
	and all charts, which arise after the $ \ell $-th blowup are put on level $ \ell $. 
	Two vertices $ v $ on level $ \ell $ and $ w $ on level $ \ell +1 $ are connected by an edge if the chart corresponding to $ w $ is one of the charts of the blowup in $ v $. 
	The unique vertex on level $ 0 $ is called the {\em root} of the tree and
	vertices on level $ \ell $, which are not connected to any vertex of a higher level, are called {\em leaves} of the tree.
	The latter correspond to the final charts of the blowup procedure. 
	In Figure~\ref{Fig:BlowupTree}, we illustrate the tree structure for a simple example.
	
	\begin{figure} 
		\[ 
		\begin{tikzpicture}
		
		\draw  (2, 1.5) -- (0,0) -- (-2, 1.5);
		\draw  (-4, 3) -- (-2, 1.5) -- (0,3);
		\draw  (-2, 4.5) -- (0,3) -- (2,4.5);
		
		\node at (0,0) {\fcolorbox{black}{white}{$ f = y^2 - x^3 $}};
		\node at (-2, 1.5) {\fcolorbox{black}{white}{$ D_+(X) : \,  x^2 (y^2 - x) $}};
		\node at (2, 1.5) {\fcolorbox{black}{white}{$ D_+(Y) : \, y^2 (1 - x^3 y) $}};
		\node at (-4, 3) {\fcolorbox{black}{white}{$ D_+(X) : \, x^3 ( x y^2 - 1) $}};
		\node at (0, 3) {\fcolorbox{black}{white}{$ D_+ (Y) : \, x^2 y^3 (y - x) $}};
		\node at (-2, 4.5) {\fcolorbox{black}{white}{$ D_+ (X) : \, x^6 y^3 (y - 1) $}};
		\node at (2, 4.5) {\fcolorbox{black}{white}{$ D_+ (Y) : \, x^2 y^6 (1 - x) $}};

		\node at (5,0) {(0)}; 
		\node at (5,1.5) {(1)}; 
		\node at (5,3) {(2)}; 
		\node at (5,4.5) {(3)}; 
		 
		\end{tikzpicture}
		\]
		\caption{Example for the tree structure of a blowup process for $ f = y^2 - x^3 $.
		In the boxes we indicate, which of the charts of the blowup we are considering, 
		and we provide the total transform of $ f $. 
		We abuse notation and denote the coordinates in each chart by $ (x,y) $.
		In every case, where a blowup is performed, the ideal of the center is $ \langle x, y \rangle $. 
		The number $ (\ell) $ on the right marks the level.
		There are two leaves on level 3 and one leaf each on level 2 and 1.}
		\label{Fig:BlowupTree} 
	\end{figure}
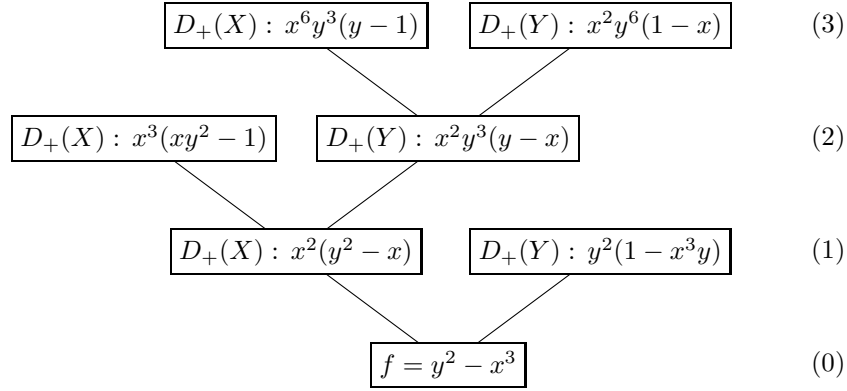
	
	The number of charts is delimited by the number of charts, which we newly create after a blowup, 
	and by the longest path from the root to any leaf.
	The first of these numbers is determined by the codimension of the center. 
	This provides the following bounds:
	
	\begin{equation} 
	\label{eq:bounds_charts_one_bu}
	\def\arraystretch{1.5}
		\begin{array}{|l|c|}
		\hline 
			\mbox{mode}  
			&\mbox{maximal codimension of a possible center}
			\\
			\hline 
			1 \mbox{ (max.ord.)}
			&
			n = 
			\mbox{(number of variables)}
			\\
			\hline 
			2 \mbox{ (codim.2)}
			&
			2
			\\
			\hline 
			3 \mbox{ (min.codim.)}
			&
			4
			\\
			\hline 
			4 \mbox{ (exc.)}
			&
			4
			\\
			\hline 
		\end{array}
	\end{equation}
\end{Rk}

Using the invariant, which we introduced to prove the termination of the respective variant for monomialization, 
we can bound the length of the longest path of the resulting tree of blowups. 

\begin{Lem}
	\label{Lem:longest_path}
	Let $ g = \ux^A - \cc \ux^B \in K[\ux] = K [x_1, \ldots, x_n] $
	with $ \cc \in K^\times $, $ A, B \in \IZ_{\ge 0}^n $ 
	such that $ A_i B_i = 0 $ for all $ i \in \{ 1, \ldots, n \} $.
	The following are upper bounds for the longest path
	from the root to any leaf in the blowup tree of the respective mode:	
	\begin{equation} 
	\label{eq:bounds_tree}
	\def\arraystretch{1.5}
	\begin{array}{|l|c|}
	\hline 
	\mbox{mode}  
	&\mbox{upper bound for the longest path}
	\\
	\hline 
	1 \mbox{ (max.ord.)}
	&
	\displaystyle 
	2^{m-1} M + m - 1 
	- \sum_{\ell=1}^{m-1} 2^{m-\ell-1} (m - \ell + 1)
	\\
	\hline 
	2 \mbox{ (codim.2)}
	&
	(\alpha + \beta -4)(n-1) + \fa + \fb + 1
	\\
	\hline 
	3 \mbox{ (min.codim.)}
	&
	(\alpha + \beta -4)(n-1) + \fa + \fb + 1
	\\
	\hline 
	4 \mbox{ (exc.)}
	&
	(\alpha + \beta -4)(n-1) + \fa + \fb + 1
	\\
	\hline 
	\end{array}
	\end{equation}
	where 
	\[ 
	\begin{array}{ll} 
	m :=\min\{ |A|,|B|\}, 
	&
	M:= \max\{ |A|,|B|\}, 
	\\
	\alpha := \max\{A_i \mid i \in \{1,\dots,n\}\}, 
	&
	\fa :=  \#\{i \in \{1,\dots,n\} \mid A_i = \alpha(g) \},
	\\
	\beta := \max\{B_i \mid i \in \{1,\dots,n\}\},
	&
	\fb :=  \#\{i \in \{1,\dots,n\} \mid B_i = \beta(g)\}.
	\end{array} 
	\]
\end{Lem}

\begin{proof}
Recall that for centers contained in maximal order locus of $ g  $ ($\mbox{mode}=1$), we introduced
$
	\inv(g) = \big(   \min\{ |A|, |B| \}, \, \max \{ |A|, |B| \}  \big) \in \IZ_{\ge 0}^2 
$
as measure for the complexity of the singularity (Definition~\ref{Def:inv}), which strictly decreases
with respect to the lexicographical ordering on $ \IZ_{\ge 0}^2 $
after each blowup.
We have 
$  
	\inv(g) = (m,M).
$
Observe that $ m \leq M $.
Let $ (m',M') $ be the value of $ \inv(g') $ in a chart after the blowup in a center contained in the locus of maximal order of $ g $. 
The proof of Proposition~\ref{Prop:max_ord_terminates} provides that either
\begin{enumerate}
	\item[$(i)$] $ (m',M') \leq_\lex (m-1, 2M -m) $, or 
	\item[$(ii)$] $	(m',M') \leq_\lex (m,M-1) $.
\end{enumerate}
If we are $ k_1 $ times in case $ (ii) $ and then once in $ (i) $, 
we obtain that the value of $ \inv(.) $ is bounded by
\[
	(m-1, \, 2 (M-k_1) - m ).
\]
Note that $ k_1 \in \{ 0, \ldots,  M - m \} $ since we cannot have $ M-1 < m $ in $ (ii) $.
At this stage, we performed $ k_1 + 1 $ blowups. 
Iterating this, the upper bound for the invariant becomes
\[
	\Big( 
	m-s, \, 
	2^{s} M - \sum_{\ell = 1}^{s} 2^{s-\ell+1} k_\ell - \sum_{\ell=1}^s 2^{s-\ell} (m - \ell + 1) 
	\Big) 
\] 
after $ s $ steps,
where,
for $ \ell \in \{ 1, \ldots, s\} $, 
\[  
	k_\ell \in \{  0, \ldots, 2^{\ell-1} M - \sum_{q=1}^{\ell - 1} 2^{\ell - q } k_q - \sum_{q=1}^{\ell - 1} 2^{\ell-1 - q} (m-q+1) - (m-\ell+1) \} ,
\]
is the number of times that we are in case $ (ii) $ until the first entry drops from $ m - \ell + 1 $ to $ m -\ell $ in case $ (i) $. 
The number of blowups up to this point is $ k_1 + \ldots + k_s + s $. 

If $ s = m -1 $, the first entry of the bound for $\inv(.) $ is $ 1 $. In particular, we can only be in case $ (ii) $ for the remaining decreases if we want to determine the maximal length of a path in the blowup tree. 
Hence, the second entry for the bound determines the number of blowups remaining. 
In total, we obtain the bound
\[
	\sum_{\ell=1}^{m-1} k_\ell + (m-1) +
	2^{m-1} M - \sum_{\ell = 1}^{m-1} 2^{m-\ell} k_\ell - \sum_{\ell=1}^{m-1} 2^{m-\ell-1} (m - \ell + 1)
	=
\]
\[
   =
   2^{m-1} M + m - 1 
	- \sum_{\ell = 1}^{m-1} (2^{m-\ell}-1) k_\ell- \sum_{\ell=1}^{m-1} 2^{m-\ell-1} (m - \ell + 1).
\]
We aim to maximize the bound for the number of blowups.
The only variation in the expression are the numbers $ k_1, \ldots, k_{m-1} $. 
Since $ 2^{m-\ell}-1 \geq 1 $
the maximum is obtained if $ k_\ell = 0 $, for all $ \ell \in \{ 1, \ldots, m- 1 \} $.
In conclusion, we have proven the bound of \eqref{eq:bounds_tree} for $ \mbox{mode} = 1 $.

\smallskip 

Let us consider the remaining three cases ($ \mbox{mode}\in \{ 2,3,4\} $).
For each of them, we used
$
	\iota(g) = \left( \alpha(g), \mathfrak{a}(g), \beta(g), \mathfrak{b}(g)\right)
	 \in \mathbb{Z}_{\ge 0}^4
$
of Definition~\ref{Def:Invariant_codim_2}
to measure the improvement of the singularity after the blowup following the respective strategy.
We have $ \iota(g) = ( \alpha,\fa,\beta,\fb) $. 
Note that $ \fa + \fb \leq n $. 
Let us fix the value of $ \mbox{mode}\in \{ 2,3,4\} $.

Let $ (\alpha',\fa',\beta',\fb') $ be the value of $ \iota(g') $ after the blowup in the center,
which is determined by the strategy given by $ \mbox{mode} $.
Let $ j \in \{ 1, \ldots, n \} $ be such that we are in the $ X_j $-chart.
By
Propositions~\ref{Prop:codim2_terminates},~\ref{Prop:singlocus_terminates}, and Corollary~\ref{Cor:except_singlocus_terminates},
we have $ (\alpha',\fa',\beta',\fb') <_\lex (\alpha,\fa,\beta,\fb) $.
The decrease can be made more precise depending on which chart we are.  
	Since $ g = \ux^A - \cc \ux^B $ and $ A_i B_i = 0 $ for all $ i \in \{ 1, \ldots, n \} $, 
	we either have $ A_j \neq 0 $ (case $ (I) $) or $ B_j \neq 0 $ (case $ (II) $).
	In these two cases, the proofs of 
	Propositions~\ref{Prop:codim2_terminates},~\ref{Prop:singlocus_terminates}, and Corollary~\ref{Cor:except_singlocus_terminates} provide that we have:
\begin{enumerate}
	\item[$(I)$] 
	$ 
	\left\{ 
	\begin{array}{lll} (\alpha',\fa',\beta',\fb') 
	=
    (\alpha,\fa-1,\beta, \fb),  
    & &
    \text{if } \fa > 1; 
    \\
    (\alpha',\fa',\beta',\fb') 
    =
    (\alpha',\fa',\beta, \fb) 
    ,
    &
    \mbox{with } \alpha' < \alpha,
    & 
    \text{if } \fa = 1;    
    \end{array} 
    \right. 
    $
    
    \medskip 
    
	\item[$(II)$] 
	$ 
	\left\{ 
	\begin{array}{lll} (\alpha',\fa',\beta',\fb') 
	= 
    (\alpha,\fa,\beta, \fb-1),  
    &&
    \text{if } \fb > 1 
    \\
    (\alpha',\fa',\beta',\fb') 
    =
    (\alpha,\fa,\beta',\fb') 
 	, 
    & 
    \mbox{with } \beta' < \beta,
    &
    \text{if } \fb = 1    .
    \end{array} 
    \right. 
    $
\end{enumerate}
Note that $ \fa' \leq n - \fb  $ in $ (I) $ and $ \fb' \leq n - \fa  $ in $ (II) $.
Since in $ (I) $ (resp.~$(II)$), the last two (resp.~first two) entries remain the same,
the operations determined by $(I)$ and $(II)$ are independent of each other. 

Due to the bound of $ \fa' $ and $ \fb' $,
we get the longest path in the resolution tree
if both the value of $ \fa(.) $ and $ \fb(.) $ decrease to the value $ 1 $ first. 
This is achieved after $ \fa + \fb - 2 $ blowups following the respective monomialization procedure. 
We denote the strict transform of $ g $ at this step by $ g'' $,
which is obtained after factoring the monomial part from the total transform of $ g $.
Notice that we have 
$ 
	\iota(g'') = (\alpha, 1, \beta, 1 ). 
$ 
After the next blowup the invariant is at most 
$ (\alpha -1, n -1 , \beta, 1 ) $ (in case $(I)$) 
resp.~$ (\alpha, 1 , \beta-1, n-1 ) $ (in case $(II)$).
If the maximum is attained, 
then we get after $ n -2 $ further blowups the value 
$ (\alpha -1, 1 , \beta, 1 ) $ 
resp.~$ (\alpha, 1 , \beta-1, 1 ) $.
Therefore, after at most $ \fa + \fb  + (\alpha - 2 )(n-1) + (\beta - 2)(n-1) $ blowups,
the value of $ \iota (.) $ becomes $ ( 1, 1, 1 , 1) $.
The latter means that the total transform of $ g $ is of the form $ \ux^C ( x_1 - \rho x_2) $, for some $ C \in \IZ_{\ge 0}^n $ and appropriately chosen variables. 
After one more blowup, the total transform of $ g $ 
fulfills condition~\eqref{eq:ex_loc_mon_cond} and thus 
is locally monomial
and the assertion follows.
\end{proof}

As an immediate consequence of \eqref{eq:bounds_charts_one_bu} and Lemma~\ref{Lem:longest_path}, we obtain:
\begin{Cor}
	Let $ g = \ux^A - \cc \ux^B \in K[x_1, \ldots, x_n] $
	with $ \cc \in K^\times $ and $ A, B \in \IZ_{\ge 0}^n $ 
	such that $ A_i B_i = 0 $ for all $ i \in \{ 1, \ldots, n \} $.
	Using the notation of Lemma~\ref{Lem:longest_path},
	the following are upper bounds for the number of charts in the respective variant for monomializing $ g $:	
	\[ 
	\def\arraystretch{1.5}
	\begin{array}{|l|c|}
	\hline 
	\mbox{mode}  
	&\mbox{upper bound for the number of charts}
	\\
	\hline 
	1 \mbox{ (max.ord.)}
	&
	\displaystyle 
	n^{d(m,M)}
	\\
	&
	\displaystyle 
	\mbox{for }
	d(m,M) :=
	2^{m-1} M + m - 1 
	- \sum_{\ell=1}^{m-1} 2^{m-\ell-1} (m - \ell + 1)
	\\
	\hline 
	2 \mbox{ (codim.2)}
	&
	2^{(\alpha + \beta -4)(n-1) + \fa + \fb + 1}
	\\
	\hline 
	3 \mbox{ (min.codim.)}
	&
	4^{(\alpha + \beta -4)(n-1) + \fa + \fb + 1}
	\\
	\hline 
	4 \mbox{ (exc.)}
	&
	4^{(\alpha + \beta -4)(n-1) + \fa + \fb + 1}
	\\
	\hline 
	\end{array}
	\]	
\end{Cor}

Moreover, this worst case number of charts yields to a non-polynomial time algorithm in the number of variables and the degree of the binomial. The running time results as a product of the running time per chart times the number of charts and for each procedure the numer of charts is potentially exponential in $\alpha, \beta, \mathfrak{a}, \mathfrak{b}$ resp. $m,M$, where $\alpha$ and $\beta$ are bounded by the degree of the binomial, $\mathfrak{a}$ and $\mathfrak{b}$ are bounded by the number of variables and $M$ and $m$ are bounded by the degree of $g$, too.

Of course, the upper bounds are quite rough
and the concrete number of blowups can be far smaller for explicit examples.
For example, in the variant, where we choose the centers in the locus of maximal order, the codimension of the center is not necessarily always $ n $ in every blowup.


\medskip

Let us come to the study of explicit examples.
In Figures~\ref{Fig:Ex1} and \ref{Fig:Ex2}, we provide several examples,
where we consider the number of leaves and the number of total charts for each method for choosing the center. 
All examples are computed via an explicit implementation in {\sc Singular} of the algorithms described in the previous sections and the base field is always $ \mathbb{Q} $.
In the following, we discuss patterns, which can be observed in the examples, 
	and we provide some indications for the noticed behavior of the different methods. 
	Nonetheless, we do not provide rigorous proofs for the patterns in general.

\begin{figure} 
\[ 
\small
\def\arraystretch{1.5}
\begin{array}{|l|l|c|c|c|c|}
\hline 
& \mbox{binomial}	
& \mbox{max.ord.}
& \mbox{codim.2}
& \mbox{min.codim}
& \mbox{exc.}
\\
\hline 
1. 
& 
x_1x_2-x_3x_4
&  
4 \,  /  \, 5 
&  
2 \,  /  \, 3 
&  
4 \,  /  \, 5 
& 
4 \,  /  \, 5 
\\
\hline 
2. 
& 
x_1x_2-x_3x_4x_5
&  
10 \,  /  \, 13
&  
3 \,  /  \, 5
&  
10 \,  /  \, 13 
& 
10 \,  /  \, 13
\\
\hline  
3. 
& 
x_1x_2-x_3x_4x_5x_6
&  
22 \,  /  \, 29
&  
4 \,  /  \, 7
&  
22 \,  /  \, 29 
& 
22 \,  /  \, 29
\\
\hline  
4. 
&
x_1x_2x_3-x_4x_5x_6
&  
60 \,  /  \, 79 
&  
6 \,  /  \, 11 
&  
40 \,  /  \, 53 
& 
40 \,  /  \, 53 
\\
\hline 
5. 
& 
x_1x_2x_3-x_4x_5x_6x_7
&  
246 \,  /  \, 325
&  
10 \,  /  \, 19
&  
124 \,  /  \, 165
& 
124 \,  /  \, 165
\\
\hline
6. 
& 
x_1x_2x_3-x_4 \cdots x_8
&  
876 \,  /  \, 1.159
&  
15 \,  /  \, 29
&  
340 \,  /  \, 453
& 
340 \,  /  \, 453
\\
\hline  
7.
&
x_1 \cdots x_4-x_5\cdots x_8
&
1.968 \, / \, 2.601 
&
20 \, / \, 39
&
496 \, / \, 661 
&
496 \, / \, 661 
\\
\hline 
9. 
& 
x_1\cdots x_4-x_5 \cdots x_9
&  
11.376 \,  /  \, 15.041 
&  
35 \,  /  \, 69
&  
1.672 \,  /  \, 2.229
& 
124 \,  /  \, 165
\\
\hline 
10.
&
x_1\cdots x_5-x_6\cdots x_{10}
&
113.760 \, / \,  150.411 
&
 70\, / \, 139 
&
 6.688\, / \, 8.917 
&
6.688 \, / \, 8.917 
\\
\hline 
\hline
11.
&
x_1x_2-x_3^2
&
3 \, / \, 4
&
3 \, / \, 5
&
3 \, / \, 4
&
3 \, / \, 4
\\
\hline

12.
&
x_1x_2x_3-x_4^2
&
7 \, / \, 10
&
7 \, / \, 13
&
7 \, / \, 10
&
7 \, / \, 10 
\\
\hline 
13.
&
x_1x_2x_3x_4-x_5^2
&
15 \, / \, 22
&
15 \, / \, 29
&
15 \, / \, 22
&
15 \, / \, 22
\\
\hline
14.
&
x_1x_2x_3-x_4^4
&
21 \, / \, 33
&
21 \, / \, 41
&
21 \, / \, 31
&
21 \, / \, 38
\\
\hline 
15.
&
x_1x_2x_3x_4-x_5^4
&
85 \, / \, 134
&
85 \, / \, 169
&
85 \, / \, 127
&
85 \, / \, 162
\\
\hline 
16.
&
x_1 \cdots x_5-x_6^4
&
341 \, / \, 538
&
341 \, / \, 681
&
341 \, / \, 511
&
341 \, / \, 666
\\
\hline 
17.
&
x_1\cdots x_6-x_7^4
&
1.365 \, / \, 2.154 
&
1.365 \, / \, 2.729 
&
1.365 \, / \, 2.047
&
1.365 \, / \, 2.698
\\
\hline 
\hline
18.
&
x_1x_2-x_3^3
&
4 \, / \, 6
&
4 \, / \, 7
&
4 \, / \, 6
&
4 \, / \, 6

\\
\hline 

19.
&
x_1x_2x_3-x_4^3
&
13 \, / \, 20
&
13 \, / \, 25
&
19 \, / \, 30
&
13 \, / \, 22
\\
\hline 
20.
&
x_1x_2x_3x_4-x_5^3
&
40 \, / \, 62
&
40 \, / \, 79
&
104 \, / \, 164
&
40 \, / \, 72

\\
\hline
21.
&
x_1x_2x_3-x_4^5
&
31 \, / \, 47
&
31 \, / \, 61
&
43 \, / \, 67
&
31 \, / \, 58
\\
\hline 
22.
&
x_1x_2x_3x_4-x_5^5
&
236 \, / \, 364
&
156 \, / \, 311
&
364 \, / \, 565
&
156 \, / \, 304
\\
\hline 
23.
&
x_1\cdots x_5-x_6^5
&
1.181 \, / \, 1.822
&
781 \, / \, 1.561
&
3.381 \, / \, 5.223
&
781 \, / \, 1.546
\\
\hline 
24.
&
x_1\cdots x_6-x_7^5
&
5.906 \, / \, 9.112 
&
3.906 \, / \, 7.811 
&
32.782 \, / \, 50.521
&
3.906 \, / \, 7.780
\\
\hline

\end{array}
\]
\caption{List of examples.
	In the last four columns, the entries are 
	``the number of leaves/total number of charts".}
\label{Fig:Ex1} 
\end{figure}

\begin{example}[Figure~\ref{Fig:Ex1}, Examples 1--10]
	In the first block of examples, all exponents of the starting binomial are one. 
	This has the effect that 
	the variant choosing centers of codimension two has a strong advantage, as can be easily seen in the number of charts. 
	The reason for this is that the codimension of the centers in the other methods is very large, 
	but the effect of the blowup is not much different than with a center of codimension two.
	
	Let us illustrate this for $ g = x_1 x_2 x_3 - x_4 x_5 x_6 $.
	The codimension two center, which will choose is $ V (x_1, x_4 )$.
	In the $ X_1 $-chart of the blowup, the strict transform of $ g $ is 
	$ g' = x_2 x_3 - x_4 x_5 x_6 $.
	(Here, we abuse notation and denote the variables in the chart of the blowup also by $ x_1, \ldots, x_6 $.)
	In the $ X_4 $-chart, we obtain the strict transform $ x_1 x_2 x_3 - x_5 x_6 $. 
	
	On the other hand, the variant choosing a center in the locus of maximal order determines the origin $ V (x_1, \ldots, x_6 ) $ 
	as the unique center. 
	In the $ X_1 $-chart of the corresponding blowup, the strict transform of $ g $ is 
	$ g' = x_2 x_3 - x_4 x_5 x_6 $.
	This is the same as the one before.
	In the remaining $ 5 $ charts, the strict transforms are of the same (up to renaming the variables).
	In contrast to the codimension two center, we have $ 6 $ different charts instead of only $ 2 $.
	
	For the other two methods, the center is $ V (x_1, x_2, x_4, x_5 ) $ and the analogous behavior appears, as the reader may verify.  
\end{example}

\begin{example}[Figure~\ref{Fig:Ex1}, Examples 11--17]
	\label{Ex:Block2} 
	In the second block of examples, all binomials have a term $ x_n^{2k} $ appearing and the other monomial is of the form as in Examples 1--10.
	The number of leaves is in all cases the same,  the total number of charts varies.
	The variant which seems to be most efficient is the one, where we choose the centers in the singular locus with minimal codimension.
	Note that the centers in this case are all of codimension three. 
	
	If $ k = 1 $, the maximal order is two. 
	Therefore, the centers for $ \mbox{mode} \in \{ 1, 3, 4 \} $ coincide. 
	Since the codimension of the center is three in these cases, 
	none of the variables remaining in the strict transforms are exceptional.
	Moreover, already for $ x_1 x_2 - x_3^2 $ one can verify that the codimension two centers requires more blowups to monomialize than the other methods.
	In the latter, the blowup with center $ V(x_1, x_2, x_3 ) $ is sufficient.
	
	This changes slightly for $ k = 2 $ and the differences in the number of total charts increases.
	The variant for $ \mbox{mode} = 3 $ remains the most efficient one.
	For example, the maximal order locus of $ x_1 x_2 x_3 - x_4^4 = 0  $ is $ V (x_1, x_2, x_3, x_4) $,
	while $ V(x_1, x_2, x_4 ) $ is the chosen center of the minimal codimension contained in the singular locus. 
\end{example}

\begin{example}[Figure~\ref{Fig:Ex1}, Examples 18--24]
	The last block of examples in Figure~\ref{Fig:Ex1} is of the same form as the previous one, but we have $ x_n^{2k+1} $ as a monomial.
	As one can see in the numbers, the behavior changes, except for Example 18.
	
	If the maximal order of the binomial is $ \leq 3 $, 
	then the variant, which chooses the center in the locus of maximal order, is the best choice. 
	One can verify that in this method, the behavior is the same as in the previous block.  
	
	One of the reasons for the large numbers for $\mbox{mode}=3$ 
	(centers of minimal codimension contained in the singular locus)
	is that the codimension is three at the beginning. 
	For example, if $ k = 1 $, 
	the strict transform of the binomial in the $ X_n $-chart of the first blowup is $ x_1 \cdots x_{n-1} - x_n $.
	This is not transversal to the exceptional divisor, which is given by $ x_n = 0 $ in this chart. 
	Therefore $ n - 1 $ more blowups are necessary in this chart.
	Moreover, in the $ X_1 $-chart, we obtain $ x_2 \cdots x_{n-1} - x_1 x_n^3 $ as strict transform, so $ x_1 $ just switch the side in the binomial.
	In contrast to this, the variable $ x_1 $ disappears in the $ X_1 $-chart of the blowup in $ V(x_1,x_2,x_3, x_n) $ 
	(which is the center for $ \mbox{mode}=1 $). 
	
	The phenomenon that a variable is switching sides also appears 
	for the other two methods ($\mbox{mode}\in \{ 2, 4 \}$), but it has less impact for the codimension two centers as we are creating less charts, where the mentioned blowups arise. 
	On the other hand, for examples of this kind of larger maximal order, 
	the codimension 2 center become more efficient
	and the last variant ($\mbox{mode}=4$) is slightly better.
	The reasons for the latter are the same as in Example~\ref{Ex:Block2}.
\end{example}

\begin{figure} 
\[ 
\small 
\def\arraystretch{1.5}
\begin{array}{|l|l|c|c|c|c|}
\hline 
& \mbox{binomial}	
& \mbox{max.ord.}
& \mbox{codim.2}
& \mbox{min.codim}
& \mbox{exc.}
\\
\hline 
25.
&
x_1^2-x_2^3x_3^4
&
6 \, / \, 11
&
6 \, / \, 11
&
6 \, / \, 11
&
6 \, / \, 11
\\
\hline 
26.
&
x_1^2-x_2^3x_3^4x_4^5
 &
12 \, / \, 22
&
12 \, / \, 23
&
12 \, / \, 22
&
12 \, / \, 23
\\
\hline 
27.
&
x_1^2-x_2^3x_3^4x_4^5x_5^6
 &
18 \, / \, 34
&
15 \, / \, 29
&
15 \, / \, 28
&
15 \, / \, 29
\\
\hline 
28. 
&
x_1^2-x_2^3x_3^4x_4^5x_5^6\cdots x_{11}^{12}
 &
288 \, / \, 560 
&
98 \, / \, 195 
&
98 \, / \, 180
&
98 \, / \, 195
\\
\hline
29. 
&
x_1^2-x_2^3x_3^4x_4^5x_5^6\cdots x_{16}^{17}
 &
2.304 \, / \, 4.480 
&
582 \, / \, 1.163 
&
582 \, / \, 1.036 
&
582 \, / \, 1.163
\\
\hline
\hline

30. 
&
x_1x_2-x_3^3x_4^4
&  
8 \,  /  \, 12 
&  
8 \,  /  \, 15 
&  
8 \,  /  \, 12 
& 
8 \,  /  \, 13 
\\
\hline 

31.
&
x_1x_2-x_3^3x_4^4x_5^5
&
16 \, / \, 24
&
13 \, / \, 25
&
16 \, / \, 24
&
13 \, / \, 22
\\
\hline 
 
32.
&
x_1x_2-x_3^3x_4^4x_5^5x_6^6
&
28 \, / \, 42
&
19 \, / \, 37
&
22 \, / \, 33
&
19 \, / \, 33
\\
\hline 
33.
&
x_1x_2-x_3^3x_4^4x_5^5x_6^6\cdots x_{12}^{12}
&
512 \, / \, 768	
&
76 \, / \, 151 
&
132 \, / \, 198 
&
76 \, / \, 141  
\\
\hline
34.
&
x_1x_2-x_3^3x_4^4x_5^5x_6^6\cdots x_{17}^{17}
&
4.096 \, / \, 6.144	
&
151 \, / \,  301 
&
652 \, / \,  978 
&
151 \, / \, 286  
\\
\hline
\hline 
35.
&
x_1^2-x_2^2x_3
&
2 \, / \, 3
&
2 \, / \, 3
&
2 \, / \, 3
&
2 \, / \, 3
\\
\hline 
36.
&
x_1^2-x_2^2x_3^2x_4
&
3 \, / \, 5
&
3 \, / \, 5
&
3 \, / \, 5
&
3 \, / \, 5
\\
\hline

37.
&
x_1^2-x_2^2x_3^2x_4^2x_5^2\cdots x_{11}^2x_{12}
&
11 \, / \, 21
&
11 \, / \, 21
&
11 \, / \, 21
&
11 \, / \, 21
\\
\hline
38.
&
x_1^2-x_2^2x_3^2x_4^2x_5^2\cdots x_{16}^2x_{17}
&
16 \, / \, 31
&
16 \, / \, 31
&
16 \, / \, 31
&
16 \, / \, 31
\\
\hline
\hline 
39.
&
x_1x_2^2-x_3x_4^2
&
6 \, / \, 9
&
2 \, / \, 3
&
2 \, / \, 3
&
2 \, / \, 3
\\
\hline

40.
&
x_1^2 x_2^3 - x_4 x_5^2 x_6^2
&
146 \, / \, 264
&
24 \, / \, 47
&
22 \, / \, 42
&
24 \, / \, 47
\\
\hline

41.
&
x_1x_2^2x_3^3-x_4x_5^2x_6^3
&
385 \, / \, 677 
&
26 \, / \, 51
&
16 \, / \, 29
&
16 \, / \, 29
\\
\hline

42.
&
x_1^2 x_2^2 x_3^2 -x_4x_5^2x_6^3
&
1.486 \, / \, 2.677
&
154 \, / \, 307
&
108 \, / \, 191
&
115 \, / \, 218
\\
\hline

43.
&
x_1^2 x_2^3 x_3^3 - x_4 x_5^2 x_6^2 x_7^3
&
18.702 \, / \, 34.262
&
126 \, / \, 251
&
104 \, / \, 196
&
124 \, / \, 246
\\
\hline 
44.
&
x_1x_2^2x_3^3x_4^4-x_5x_6^2x_7^3x_8^4
&
107.062 \, / \, 196.798 
&
260 \, / \, 519 
&
206 \, / \, 371
&
213 \, / \, 392
\\
\hline
\end{array}
\]
\caption{List of examples (continued).}
\label{Fig:Ex2} 
\end{figure}

\begin{example}[Figure~\ref{Fig:Ex2}, Examples 25--29]
	\label{Ex:Block?}
	Let us turn to examples,
	where the appearing exponents are larger.
	More precisely, we consider binomials of the form 
	$ x_1^2 - x_2^3 x_3^4 \cdots x_n^{n+1} $, for $ n \in \IZ_{\ge 3} $. 
	The large exponents have the effect that the centers are of codimension two at the beginning of all variants.
	Due to the exceptional divisors created there, 
	the centers for $ \mbox{mode} = 4 $ coincide with the centers in the codimension two variant ($\mbox{mode}=2$). 
	 
	For centers in the locus of maximal order and centers of minimal codimension ($\mbox{mode}\in \{ 1,3\} $),
	there appear eventually centers of higher codimension.	
	For example, in $ x_1^2 - x_2^3 \cdots x_5^6 $, we obtain after five blowups with centers of codimension two the strict transform
	$ x_1^2 - x_2 x_4 x_5^6 $ and $ ( x_2, x_3, x_4 ) $ are exceptional. 
	Our variant for choosing centers in the locus of maximal order, determines $ V ( x_1, x_2, x_4)$ as the next center,
	while, for $ \mbox{mode}=3 $, the next center is $ V (x_1, x_5) $.
	This leads to more charts in the first variant.
	Analogous to Example~\ref{Ex:Block2}, the centers of minimal codimension are slightly better in the number of total charts than the centers of codimension two, 
	but the number of leaves are the same. 
\end{example}

\begin{example}[Figure~\ref{Fig:Ex2}, Examples 30--34]
	In contrast to the previous block of examples,
	the difference between centers of codimension two and centers of higher codimension becomes more clear for binomials of the type	
	$ x_1 x_2 - x_3^3 x_4^4 \cdots x_n^{n} $, for $ n \in \IZ_{\ge 4} $. 
	As a consequence the variants with $ \mbox{mode} \in \{ 2, 4\} $ are more efficient for large $ n \gg 4 $,
	while $ \mbox{mode} =  4 $ is slightly better as the center of codimension three at the beginning provide a fast improvement.  
	On the other hand, the number of charts are larger for the remaining two variants,
	where the reason for the large numbers if $ \mbox{mode}=1 $ are the same as in Example~\ref{Ex:Block?}.  
\end{example}

\begin{example}[Figure~\ref{Fig:Ex2}, Examples 35--38]
	There are also types of binomials, for which all variants choose the same centers.
	For example, for binomials of the form 
	$ x_1^2-x_2^2\cdots x_{n-1}^2 x_n $, for some $ n \in \IZ_{\ge 3} $, 
	all centers in the monomialization procedures are of codimension two. 
	If we blowup with center $ V(x_1, x_2) $,
	the total transform of the binomial in the respective charts are 
	$ x_1^2 (1 - x_2^2\cdots x_{n-1}^2 x_n ) $
	and $ x_2^2 ( x_1^2  - x_3^2\cdots x_{n-1}^2 x_n  ) $.
	While the first chart is locally monomial, the second one is of the same form as the original binomial
	with the difference that $ x_2 $ does not appear anymore. 
	Hence, the total number of charts is $ 2 ( n-1) - 1 $ and the number of leaves is $ n- 1 $. 
\end{example}

\begin{example}[Figure~\ref{Fig:Ex2}, Examples 39--44]
	This block of examples consists of homogeneous polynomials. 
	Hence, the variant choosing the centers in the locus of maximal order 
	will first blow up the closed point, which creates many charts. 
	For increasing degree of the homogeneous binomial,
	we obtain a fast growing number of charts and leaves. 
	The other three variants are more efficient,
	where all of them choose first centers of codimension two.
	In contrast to $ \mbox{mode}= 1 $, the centers of larger codimension are more efficient than the codimension two centers towards the end of the monomialization procedure, where the appearing exponents are at most one.
	This phenomenon has already been observed in Example~\ref{Ex:Block2}. 
\end{example}

\begin{figure} 
\[ 
\small 
\def\arraystretch{1.5}
\begin{array}{|l|l|c|c|c|c|}
\hline 
& \mbox{binomial}	
& \mbox{max.ord.}
& \mbox{codim.2}
& \mbox{min.codim}
& \mbox{exc.}

\\
\hline
27.
&
x_1^2-x_2^3x_3^4x_4^5x_5^6
&
18 \, / \, 34
&
15 \, / \, 29
&
15 \, / \, 28
&
15 \, / \, 29
\\
\hline 
27.2
&
x_1^2-x_2^5x_3^3x_4^6x_5^4
 &
20 \, / \, 38
&
15 \, / \, 29
&
15 \, / \, 28
&
15 \, / \, 29
\\
\hline 
27.3
&
x_1^2-x_2^6x_3^5x_4^4x_5^3
 &
15 \, / \, 28
&
15 \, / \, 29
&
15 \, / \, 28
&
15 \, / \, 29
\\
\hline
\hline 
28. 
&
x_1^2-x_2^3x_3^4x_4^5x_5^6\cdots x_{11}^{12}
&
288 \, / \, 560 
&
98 \, / \, 195 
&
98 \, / \, 180
&
98 \, / \, 195
\\
\hline 
28.2  
&
x_1^2-x_2^{11}x_3^9\cdots x_6^3 x_7^{12}x_8^{10}\cdots x_{11}^4 
 &
414 \, / \, 812 
&
98 \, / \, 195 
&
98 \, / \, 180
&
98 \, / \, 195
\\
\hline
28.3 
&
x_1^2-x_2^{12}x_3^{11}x_4^{10}x_5^9\cdots x_{11}^3 

 &
141 \, / \, 266 
&
98 \, / \, 195 
&
98 \, / \, 180
&
98 \, / \, 195
\\
\hline

28.4  
&
x_1^2-x_2^{12}x_3^{10}\cdots x_6^4 x_7^{11}x_8^9\cdots x_{11}^3 
 &
114 \, / \, 212 
&
98 \, / \, 195 
&
98 \, / \, 180
&
98 \, / \, 195
\\
\hline 
\hline

31.
&
x_1x_2-x_3^3x_4^4x_5^5
&
16 \, / \, 24
&
13 \, / \, 25
&
16 \, / \, 24
&
13 \, / \, 22
\\
\hline 

31. 2 
&
x_1x_2-x_3^3x_4^5x_5^4
&
20 \, / \, 30
&
13 \, / \, 25
&
16 \, / \, 24
&
13 \, / \, 22
\\
\hline 
\hline 
32.
&
x_1x_2-x_3^3x_4^4x_5^5x_6^6
&
28 \, / \, 42
&
19 \, / \, 37
&
22 \, / \, 33
&
19 \, / \, 33
\\
\hline 
 
32. 2 
&
x_1x_2-x_3^6x_4^5x_5^4x_6^3
&
22 \, / \, 33
&
19 \, / \, 37
&
22 \, / \, 33
&
19 \, / \, 33
\\
\hline 
\hline 
33.
&
x_1x_2-x_3^3x_4^4x_5^5x_6^6\cdots x_{12}^{12}
&
512 \, / \, 768	
&
76 \, / \, 151 
&
132 \, / \, 198 
&
76 \, / \, 141  
\\
\hline

33.2
&
x_1x_2-x_3^{12}x_4^{11}x_5^{10}\cdots x_{12}^3
&
218 \, / \, 327	
&
76 \, / \, 151 
&
132 \, / \, 198 
&
76 \, / \, 141  
\\
\hline
\hline
38.
&
x_1^2-x_2^2x_3^2x_4^2x_5^2\cdots x_{16}^2x_{17}
&
16 \, / \, 31
&
16 \, / \, 31
&
16 \, / \, 31
&
16 \, / \, 31
\\
\hline
38.2
&
x_1^2-x_2x_3^{2}x_4^{2}x_5^{2}\cdots x_{17}^2
&
16 \, / \, 31
&
16 \, / \, 31
&
16 \, / \, 31
&
16 \, / \, 31
\\
\hline
\hline  
41.
&
x_1x_2^2x_3^3-x_4x_5^2x_6^3
&
385 \, / \, 677 
&
26 \, / \, 51
&
16 \, / \, 29
&
16 \, / \, 29
\\
\hline 
41.2
&
x_1^3x_2x_3^2-x_4^3x_5x_6^2
&
244 \, / \, 427 
&
12 \, / \, 23
&
16 \, / \, 29
&
12 \, / \, 23
\\
\hline

41.3
&
x_1^2x_2x_3^3-x_4^2x_5x_6^3
&
274 \, / \, 483 
&
19 \, / \, 37
&
14 \, / \, 25
&
14 \, / \, 25

\\
\hline
\end{array}
\]
\caption{List of examples, where the different choice within one method are considered.}
\label{Fig:Ex3}
\end{figure}

\begin{example}[Figure~\ref{Fig:Ex3}]
	As we have seen, it may appear that we have to make a choice for the center in the respective variant for monomialization.
	Let us have a glimpse into the question, 
	how different choices affect the number of charts.  
	Instead of modifying the implementations, 
	we explore this by interchanging the appearing exponents appropriately in a given example. 
	
	If $ \mbox{mode} \in \{ 2,3,4\} $, 
	the numbers do not change, except for the block of example~41 in Figure~\ref{Fig:Ex3}.
	In~41, i.e.  $ x_1 x_2^2 x_3^3 - x_4 x_5^2 x_6^3 $ 
	the numbers are larger compared to the other choices 41.2 and 41.3.
	The reason for this is that the codimension two centers 
	($ \mbox{mode}\in \{ 2,4\} $) for example 41
	are of the form $ V (x_1, x_i) $ or $ V (x_j, x_4) $
	(for $ i \in \{ 4,5,6 \} $ and $ j \in \{ 1, 2, 3 \} $) 
	at the beginning of the monomialization process.   
	Hence, the improvement of the exponent of $ x_i $, resp.~$ x_j $, is only by one and more blowups are needed.
	On the other hand, $ \mbox{mode}= 3 $ is less affected by this,
	for example, the first center for example 41, $ x_1 x_2^2 x_3^3 - x_4 x_5^2 x_6^3 $, is $ V (x_2,x_5) $.
	Example 41.3 is slightly better if $ \mbox{mode}=3 $, 
	as the first appearing powers are even.
	
	The first method (via centers contained in the locus of maximal order) varies more if we interchange the exponents.
	In the cases, where the maximal order is two,
	the number of charts is significantly larger
	if we the first exponents are odd.
	The reason for this can be seen in Example~\ref{Ex:Block?},
	where the binomial became $ x_1^2 - x_2 x_4 x_5^6 $. 
	The next center following our way of choosing the center (for $ \mbox{mode}=1 $),
	would be $ V(x_1,x_2,x_4) $.
	Hence, we create three new charts and in two of them, we have to blow up $ V(x_1, x_5 ) $ three times.
	In contrast to this, if we blowup first in $ V(x_1, x_5) $,
	we would get a smaller number of charts, 
	since we only have to blow up $ V(x_1,x_2,x_4) $ at the end, when we reach $ x_1^2 - x_2 x_4 $ as strict transform.

	Note that we interchanged the exponents only at the beginning of the monomialization process.
	In principle, one could interchange them after each blowup in order to optimize the choice of the center,
	but we do not address this here. 
\end{example}

\smallskip

In conclusion,
the approach by blowing up centers contained in the locus of maximal order provides a significant larger number of charts
than the other variants if the exponents appearing in the binomial increase. 
Most of the time, 
choosing only centers of codimension two ($ \mbox{mode}=2 $) leads to a small number of charts,
while a particular structure of the binomial may
give a small advantage to the other variants ($ \mbox{mode} \in \{ 3 , 4\} $) in some cases.
But since the advantage is only small, 
our choice for a first investigation of a local monomialization of a binomial and the data resulting from it 
is via centers of codimension two.

\smallskip

%
%

\end{document}